\documentclass[12pt,english]{article}
\usepackage{mathptmx}

\usepackage[T1]{fontenc}
\usepackage[latin9]{inputenc}
\usepackage[a4paper]{geometry}
\usepackage{babel}
\usepackage{amsmath}
\usepackage{amsthm}
\usepackage{amssymb}
\usepackage{setspace}
\usepackage[numbers,square,sort,comma,numbers]{natbib}
\usepackage[all]{xy}
\usepackage[unicode=true,
 bookmarks=true,bookmarksnumbered=false,bookmarksopen=false,
 breaklinks=false,pdfborder={0 0 1},backref=false,colorlinks=false]
 {hyperref}
\hypersetup{pdftitle={Use Lyx for typesetting: A template for empirical research projects},
 pdfauthor={First Author and Second Author}}

\makeatletter
\hypersetup{
    colorlinks=true,       
    linkcolor=blue,          
    citecolor=blue,        
    filecolor=blue,      
    urlcolor=blue,           
    breaklinks=true
}

\usepackage{booktabs}
\usepackage{longtable}
\usepackage{pdflscape}
\usepackage[square,sort,comma,numbers]{natbib}
\usepackage{amsthm}
\usepackage[title]{appendix}

\makeatother

\begin{document}
\title{\onehalfspacing{}Salter's question on the image of the Burau representation of $B_{3}$}
\author{Donsung Lee}
\date{May 30, 2024}

\maketitle
\medskip{}

\begin{abstract}
\begin{spacing}{0.9}
\noindent In 1974, Birman posed the question of under what conditions
a matrix with Laurent polynomial entries is in the image of the Burau
representation. In 1984, Squier observed that the matrices in the
image are contained in a unitary group. In 2021, Salter formulated
a specific question: whether the central quotient of the Burau image
group is the central quotient of a certain subgroup of the unitary
group. We solve this question negatively in the simplest nontrivial
case, $n=3$, algorithmically constructing a counterexample.\vspace{1cm}

\noindent \textbf{Keywords:} braid group, Burau representation\medskip{}

\noindent \textbf{Mathematics Subject Classification 2020:} 20F36,
20C99, 20E05 
\end{spacing}
\end{abstract}
\noindent $ $\theoremstyle{definition}
\theoremstyle{remark}
\newtheorem{theorem}{Theorem}[section]
\newtheorem{lemma}[theorem]{Lemma}
\newtheorem{corollary}[theorem]{Corollary}
\newtheorem{question}[theorem]{Question}
\newtheorem{remark}[theorem]{Remark}
\newtheorem{example}[theorem]{Example}
\newtheorem{definition}[theorem]{Definition}
\newtheorem{algorithm}[theorem]{Algorithm}
\newtheorem{conjecture}[theorem]{Conjecture}

\newtheorem*{ack}{Acknowledgements}
\newtheorem*{theoremA}{Theorem \textnormal{A}}
\newtheorem*{question1}{Question \textnormal{1}}
\newtheorem*{question2}{Question \textnormal{2}}
\newtheorem*{claim1}{Claim \textnormal{1}}
\newtheorem*{proofclaim1}{Proof of Claim \textnormal{1}}
\newtheorem*{claim2}{Claim \textnormal{2}}
\newtheorem*{proofclaim2}{Proof of Claim \textnormal{2}}
\newtheorem*{prooftheorem49}{Proof of Theorem \textnormal{4.9}}
\newtheorem*{proofofa}{Proof of \textnormal{(a)}}
\newtheorem*{proofofb}{Proof of \textnormal{(b)}}
\newtheorem*{prooftheorem11}{Proof of Theorem \textnormal{1.1}}

\section{Introduction}

The \emph{Burau representation} is a fundamental linear representation
for \emph{braid groups} $B_{n}$, which has the generators $\sigma_{1},\,\sigma_{2},\cdots,\,\sigma_{n-1}$.
For each integer $i$ such that $1\le i\le n-1$, the representation
$\beta_{n}$ is given explicitly by
\begin{center}
$\sigma_{i}\mapsto\left(\begin{array}{cccc}
I_{i-1} & 0 & 0 & 0\\
0 & 1-t & t & 0\\
0 & 1 & 0 & 0\\
0 & 0 & 0 & I_{n-i-1}
\end{array}\right)\in\mathrm{GL}\left(n,\,\mathbb{Z}\left[t,t^{-1}\right]\right)$.
\par\end{center}

The \emph{reduced} Burau representation from $B_{n}$ to $\mathrm{GL}\left(n-1,\,\mathbb{Z}\left[t,t^{-1}\right]\right)$
can be defined, but will be not dealt with in this paper. In 1974,
Birman (\citep{MR0375281}, Research Problem 14) asked how the image
of $\beta_{n}$ is characterized in $\mathrm{GL}\left(n,\,\mathbb{Z}\left[t,t^{-1}\right]\right)$
as a group of matrices. In 1984, Squier in \citep{MR0727232} observed
that the elements in the image satisfy the \emph{unitarity}. If we
put
\begin{align*}
J_{n} & :=\left(\begin{array}{ccccc}
1 & -t^{-1} & -t^{-1} & \cdots & -t^{-1}\\
-t & 1 & -t^{-1} & \cdots & -t^{-1}\\
-t & -t & 1 & \cdots & -t^{-1}\\
\vdots & \vdots & \vdots & \ddots & \vdots\\
-t & -t & -t & \cdots & 1
\end{array}\right),
\end{align*}
then any element $A\in\beta_{n}\left(B_{n}\right)\subset\mathrm{GL}\left(n,\,\mathbb{Z}\left[t,t^{-1}\right]\right)$
satisfies the \emph{unitarity relation}
\begin{align}
\overline{A}J_{n}A^{T} & =J_{n},
\end{align}
where for a matrix $A$ with Laurent polynomial entries, $A\mapsto\overline{A}$
maps $t\mapsto t^{-1}$ for every entry. Define further for any Laurent
polynomial $f\left(t\right)\in\mathbb{Z}\left[t,\,t^{-1}\right]$,
$\overline{f\left(t\right)}:=f\left(t^{-1}\right)$. From now on,
let us consider only the $n=3$ case. Define the \emph{Burau image
group}
\begin{align*}
\mathrm{Bu} & :=\beta_{3}\left(B_{3}\right)=\left\langle \left(\begin{array}{ccc}
1-t & t & 0\\
1 & 0 & 0\\
0 & 0 & 1
\end{array}\right),\,\left(\begin{array}{ccc}
1 & 0 & 0\\
0 & 1-t & t\\
0 & 1 & 0
\end{array}\right)\right\rangle .
\end{align*}

In 1996, Dehornoy in \citep{MR1381685} derived a kind of functional
equations on the entries of matrices in $\mathrm{Bu}$, based on the
unitarity. In 2021, Salter in \citep{MR4228497} pointed out several
algebraic properties satisfied by the elements in $\mathrm{Bu}$ and
defined a group $\Gamma$ as a subgroup of the unitary group defined
by the properties. Denote by $\overline{\Gamma}$ the quotient of
$\Gamma$ by its center by the quotient map $q:\Gamma\to\overline{\Gamma}$.
Then, he posed a question whether the composition $q\circ\beta_{3}:B_{3}\to\overline{\Gamma}$
is surjective. In fact, his question was about every braid group $B_{n}$,
for an integer $n\ge2$. Since the case $n=2$ is easy to handle,
we concentrate on the simplest nontrivial case $n=3$.

The question was an attempt to answer Birman's old problem. If it
were true, we could characterize $\mathrm{Bu}$ algebraically, or
in other words, by the set of solutions of a finite set of equations.
In this paper, we negatively solve it.

\noindent \begin{theorem}

The map $q\circ\beta_{3}:B_{3}\to\overline{\Gamma}$ is not surjective.
Moreover, we have $\left[\overline{\Gamma}:q\circ\beta_{3}\left(B_{3}\right)\right]=\infty$.

\noindent \end{theorem}

It is well-known that the quotient group of $B_{3}$ by its center
is isomorphic to the modular group $\mathrm{PSL}\left(2,\,\mathbb{Z}\right)$.
One of the key ideas is to reinterpret the Burau representation $\beta_{3}$
through a map
\begin{align*}
\mu & :\mathrm{PSL}\left(2,\,\mathbb{Z}\right)\to\mathrm{PGL}\left(2,\,\mathbb{Z}\left[t,\,t^{-1},\,\left(1+t\right)^{-1}\right]\right),
\end{align*}
defined by
\begin{align*}
\mu\left[\begin{array}{cc}
1 & 0\\
1 & 1
\end{array}\right] & :=\left[\begin{array}{cc}
1 & 0\\
0 & -t
\end{array}\right],\;\mu\left[\begin{array}{cc}
1 & -1\\
0 & 1
\end{array}\right]:=\left[\begin{array}{cc}
-\frac{t^{2}}{1+t} & \frac{t}{1+t}\\
\frac{1+t+t^{2}}{1+t} & \frac{1}{1+t}
\end{array}\right],
\end{align*}
where we used the convention that a matrix in the square brakets $\left[\right]$
is included in a projective linear group, while a usual matrix is
in the round brakets $\left(\right)$. By using the map $\mu$, we
will show that the group $\overline{\Gamma}$ is isomorphic to a semidirect
product
\begin{align*}
F\rtimes\mathrm{PSL}\left(2,\,\mathbb{Z}\right),
\end{align*}
where $F$ is a free group (Theorem 3.17 and Corollary 4.11). We will
see that the surjectivity of $q\circ\beta_{3}$ is equivalent to $F$
being trivial (Lemma 4.1). This is also equivalent to the equation
\begin{align*}
f\overline{f}+\left(1+t^{-1}+t\right)g\overline{g} & =1
\end{align*}
having no solution $\left(f,\,g\right)\in\mathbb{Z}\left[t,t^{-1}\right]^{2}$
such that $g\ne0$ (Corollary 4.7). Finally, we algorithmically find
such a solution for this equation to prove Theorem 1.1 in Section
5.

The rest of this paper is organized as follows. In Section 2, we introduce
basic definitions and Salter's question. In Section 3, we introduce
two homomorphisms $\overline{\phi}$ and $\overline{\rho}$, and prove
that $\overline{\Gamma}$ is isomorphic to the semidirect product.
In Section 4, we define the quaternionic group $\mathcal{Q}$ in $\mathrm{PGL}\left(2,\,\mathbb{Q}\left[t,t^{-1}\right]\right)$
and prove a structure theorem on it to prove $F$ is free. In Section
5, we prove Theorem 1.1.

\noindent \begin{ack}

\noindent This paper is supported by \emph{National Research Foundation
of Korea (grant number 2020R1C1C1A01006819).} The author thanks Dohyeong
Kim and Sang-hyun Kim for valuable and encouraging comments.

\noindent \end{ack}

\section{Salter's Question}

We review some basic properties of the Burau representation. Define
two vectors
\begin{align*}
v & :=\left(t,\,t^{2},\,t^{3}\right),\;\overrightarrow{1}:=\left(1,\,1,\,1\right)^{T}.
\end{align*}

Every matrix $A\in\mathrm{Bu}$ satisfies the following.

\noindent 
\begin{align}
 & A\overrightarrow{1}=\overrightarrow{1},\\
 & vA=v,
\end{align}

\noindent where we followed Salter's notation, except that the matrix
transpose was taken. This convention is due to \citep{MR1731872}
and \citep{MR2435235}. Another property of the Burau image that Salter
noticed is that the image of the evaluation of $\mathrm{Bu}$ at $t=1$
is isomorphic to the symmetric group $S_{3}$, by the usual permutation
representation. Abusing notation, we identify $S_{3}$ with the representation
in $\mathrm{GL}\left(3,\,\mathbb{Z}\right)$. Then, he formally defined
the target group $\Gamma$ in \citep[Definition 2.5]{MR4228497}.

\noindent \begin{definition}

The group $\Gamma$ is defined as follows:
\begin{align*}
\Gamma & :=\left\{ A\in\mathrm{GL}\left(3,\,\mathbb{Z}\left[t,t^{-1}\right]\right)\::\:vA=v,\:A\overrightarrow{1}=\overrightarrow{1},\:\overline{A}J_{3}A^{T}=J_{3},\:A|_{t=1}\in S_{3}\right\} .
\end{align*}
\end{definition}

Let $\overline{\Gamma}$ be the quotient of $\Gamma$ by its center.
In some contexts, we will also call $\overline{\Gamma}$ the target
group, if there is no possibility for confusion. Define the map $\overline{\beta}_{3}:B_{3}\to\overline{\Gamma}$
to be the composition of the quotient $q:\Gamma\to\overline{\Gamma}$
and the Burau representation $\beta_{3}$. From now on, we omit the
subscript from $\beta_{3}$ (resp. $\overline{\beta_{3}}$) and write
it just as $\beta$ (resp. $\overline{\beta}$), since we only deal
with the case $n=3$ in this paper. Here is Salter's question on $B_{3}$
\citep[Question 1.1]{MR4228497}.

\noindent \begin{question}

Is the map $\overline{\beta}:B_{3}\to\overline{\Gamma}$ surjective?

\noindent \end{question}

For the time being, we will analyze just $\Gamma$ without taking
the quotient into account. To approach this question, we first define
a group similar to $\Gamma$.

\noindent \begin{definition}

The \emph{formal Burau group} $\mathcal{B}$ is defined as follows
\begin{align*}
\mathcal{B} & :=\left\{ A\in\mathrm{GL}\left(3,\,\mathbb{Z}\left[t,t^{-1}\right]\right)\::\:vA=v,\:A\overrightarrow{1}=\overrightarrow{1},\:\overline{A}J_{3}A^{T}=J_{3}\right\} .
\end{align*}
\end{definition}

By definition, the group $\Gamma$ is a subgroup of $\mathcal{B}$.
In fact, we will see later the omitted condition is redundant and
$\mathcal{B}=\Gamma$ (Lemma 3.15). Define a special braid $\Delta\in B_{3}$
to be
\begin{align*}
\Delta & :=\left(\sigma_{1}\sigma_{2}\sigma_{1}\right)^{2},
\end{align*}
and recall that the center of $B_{3}$ is generated by $\Delta$ \citep[Theorem 1.24]{MR2435235}.
We directly have that
\begin{align}
\beta\left(\Delta\right) & =\left(\begin{array}{ccc}
1-t+t^{3} & t-t^{2} & t^{2}-t^{3}\\
1-t & t-t^{2}+t^{3} & t^{2}-t^{3}\\
1-t & t-t^{2} & t^{2}
\end{array}\right).
\end{align}

\noindent \begin{remark}

In the definition of $\mathcal{B}$, we mention without proof that
the condition $vA=v$ may be replaced with the condition $A\beta\left(\Delta\right)=\beta\left(\Delta\right)A$.
It is a consequence of a direct computation assuming $A\overrightarrow{1}=\overrightarrow{1}$.

\noindent \end{remark}

According to the first two conditions (2) and (3) of $\mathcal{B}$,
we observe that for $A\in\mathcal{B}$, any four entries of $A$ determines
the others. Choosing the four at the upper left corner, we can write
the entries of $A$ as
\begin{align}
A & =\left(\begin{array}{ccc}
A_{11} & A_{12} & 1-A_{11}-A_{12}\\
A_{21} & A_{22} & 1-A_{21}-A_{22}\\
t^{-2}\left(1-A_{11}-tA_{21}\right) & t^{-2}\left(t-A_{12}-tA_{22}\right) & *
\end{array}\right),
\end{align}
where $*=t^{-2}\left(-1-t+t^{2}+A_{11}+A_{12}+tA_{21}+tA_{22}\right).$

Another remarkable property of $A\in\mathcal{B}$ is on its determinant.
Since we defined $\mathcal{B}$ as a subgroup of $\mathrm{GL}\left(3,\,\mathbb{Z}\left[t,t^{-1}\right]\right)$,
$\det\left(A\right)$ must be a unit of the base ring $\mathbb{Z}\left[t,t^{-1}\right]$.
Thus, the only possible candidate is $\pm t^{k}$ for some integer
$k$. In fact, we will show later that $\left(-t\right)^{k}$ for
some integer $k$ remains possible (Lemma 3.10).

\section{Target Group as a Semidirect Product}

In Section 1, we introduced a map 
\begin{align*}
\mu & :\mathrm{PSL}\left(2,\,\mathbb{Z}\right)\to\mathrm{PGL}\left(2,\,\mathbb{Z}\left[t,\,t^{-1},\,\left(1+t\right)^{-1}\right]\right),
\end{align*}
defined by
\begin{align*}
\mu\left[\begin{array}{cc}
1 & 0\\
1 & 1
\end{array}\right] & :=\left[\begin{array}{cc}
1 & 0\\
0 & -t
\end{array}\right],\;\mu\left[\begin{array}{cc}
1 & -1\\
0 & 1
\end{array}\right]:=\left[\begin{array}{cc}
-\frac{t^{2}}{1+t} & \frac{t}{1+t}\\
\frac{1+t+t^{2}}{1+t} & \frac{1}{1+t}
\end{array}\right].
\end{align*}

Put $s_{2}:=\left[\begin{array}{cc}
1 & 0\\
1 & 1
\end{array}\right]$ and $s_{1}:=\left[\begin{array}{cc}
1 & -1\\
0 & 1
\end{array}\right]$. Then, the well-definedness of $\mu$ is direct by computing the
following two relations of the modular group
\begin{align*}
s_{1}s_{2}s_{1} & =s_{2}s_{1}s_{2},\;\left(s_{1}s_{2}s_{1}\right)^{2}=1.
\end{align*}

The goal of this section is to describe the structure of Salter's
target group $\overline{\Gamma}$ by using the map $\mu$. Concretely,
we prove that
\begin{align*}
\overline{\Gamma} & \cong F\rtimes\mathrm{PSL}\left(2,\,\mathbb{Z}\right),
\end{align*}
which is given by Theorem 3.17. The key idea is to reduce the rank
by constructing the homomorphisms from the formal Burau group $\mathcal{B}$.

Let us look back to the unitarity $\overline{A}J_{3}A^{T}=J_{3}$.
A natural way to use this relation is to \emph{$*$}-diagonalize the
interposed matrix $J_{3}$. Let us define a matrix as
\begin{align*}
\xi & :=\left(\begin{array}{ccc}
t+t^{2} & 0 & -1\\
-1 & t & -1\\
-1 & -1 & -1
\end{array}\right).
\end{align*}

Then, if we define a matrix
\begin{align*}
D & :=\left(\overline{\xi}\right)^{-1}J_{3}\left(\xi^{T}\right)^{-1},
\end{align*}
by direct computation, we get a diagonal matrix

\begin{equation}
D=\left(\begin{array}{ccc}
\frac{t}{1+t+t^{2}} & 0 & 0\\
0 & 1 & 0\\
0 & 0 & -\frac{t}{1+t+t^{2}}
\end{array}\right).
\end{equation}

\noindent \begin{remark}

The introduction of the matrix $\xi$ may seem out of nowhere. We
briefly comment on how we found this. There is an established method
to the $*$-diagonalization of a matrix (for example, see \citep{MR2571854}).
We used the matrix $\Omega_{3}$ in \citep[p. 95]{MR2435235} defining
the equivalent unitarity relation to that by $J_{3}$ to compute the
$*$-cosquare $\overline{\Omega_{3}}\left(\Omega_{3}^{T}\right)^{-1}$,
and diagonalized it. The matrix $\xi$ is constructed by a proper
linear combination of the eigenvectors. We should mention that Chen,
Lubotzky and Tiep in \citep{2308.14302} already found a $*$-diagonalized
unitarity relation.

\noindent \end{remark}

Consider a matrix $A\in\mathcal{B}$. By (5), we express its entries
as
\begin{align*}
A & =\left(\begin{array}{ccc}
A_{11} & A_{12} & 1-A_{11}-A_{12}\\
A_{21} & A_{22} & 1-A_{21}-A_{22}\\
t^{-2}\left(1-A_{11}-tA_{21}\right) & t^{-2}\left(t-A_{12}-tA_{22}\right) & *
\end{array}\right).
\end{align*}

To apply the unitarity (1), for integers $1\le k,\,l\le2$, define
a Laurent polynomial
\begin{align*}
f_{kl} & :=\left(A\xi\right)_{kl}.
\end{align*}

Then, from the definition of $\xi_{3}$, we have
\begin{align}
f_{k1} & =A_{k1}\left(1+t+t^{2}\right)-1,\;f_{k2}=A_{k1}+A_{k2}\left(1+t\right)-1,\;k=1,2.
\end{align}

\noindent \begin{lemma}
\begin{description}
\item [{(a)}] Each matrix $A\in\mathcal{B}$, the determinant is expressed
by $\left\{ f_{kl}\right\} _{1\le k,\,l\le2}$ as
\begin{align}
\det\left(A\right) & =\frac{f_{11}f_{22}-f_{12}f_{21}}{t^{2}\left(1+t\right)}.
\end{align}
\item [{(b)}] The Laurent polynomials $\left\{ f_{kl}\right\} _{1\le k,\,l\le2}$
satisfies the three functional equations
\begin{align}
 & f_{k1}\overline{f_{k1}}+\left(1+t+t^{-1}\right)f_{k2}\overline{f_{k2}}=\frac{\left(1+t\right)^{2}}{t},\;k=1,2,\\
 & 1+t+t^{2}\left(\overline{f_{11}}f_{21}+\left(1+t+t^{-1}\right)\overline{f_{12}}f_{22}\right)=0.
\end{align}
\end{description}
\noindent \end{lemma}
\begin{proof}

\noindent Rewrite the unitarity satisfied by $A\in\mathcal{B}$ as
\begin{align*}
\overline{A\xi}D\left(A\xi\right)^{T} & =J_{3}.
\end{align*}

Then, the properties desired are direct from computation from (6)
and (7). $\qedhere$

\noindent \end{proof}

By observing Lemma 3.2 (a) and (b), we can infer that the polynomials
$\left\{ f_{kl}\right\} _{1\le k,\,l\le2}$ correspond to the entries
of a $2\times2$ matrix. Recall that we already have two well-known
homomorphisms from the Burau image group $\mathrm{Bu}$ to $2\times2$
linear groups. At first, there is an isomorphism from the unreduced
Burau image in $\mathrm{GL}\left(3,\mathbb{Z}\left[t,t^{-1}\right]\right)$
to the reduced Burau image in $\mathrm{GL}\left(2,\mathbb{Z}\left[t,t^{-1}\right]\right)$.
Next, there is a quotient homomorphism from the reduced Burau image
to the modular group $\mathrm{PSL}\left(2,\mathbb{Z}\right)$, induced
by the evaluation map at $t\mapsto-1$. We will reconstruct these
two homomorphisms in terms of $\left\{ f_{kl}\right\} _{1\le k,\,l\le2}$.

The first step is to rewrite the entries of a matrix $A\in\mathcal{B}$
in (5) by using $f_{ij}$ as
\begin{align}
A & =\left(\begin{array}{ccc}
\frac{1+f_{11}}{1+t+t^{2}} & \frac{-f_{11}+t\left(1+t\right)+f_{12}\left(1+t+t^{2}\right)}{\left(1+t\right)\left(1+t+t^{2}\right)} & *\\
\frac{1+f_{21}}{1+t+t^{2}} & \frac{-f_{21}+t\left(1+t\right)+f_{22}\left(1+t+t^{2}\right)}{\left(1+t\right)\left(1+t+t^{2}\right)} & *\\
* & * & *
\end{array}\right).
\end{align}

For integers $1\le i,j\le2$, let us introduce a new substitution
\begin{align*}
g_{ij} & :=\frac{f_{ij}}{t\left(1+t\right)},
\end{align*}
to reduce the $\left(1+t\right)$ factor in Lemma 3.2 (a). Then, $A$
can be written as
\begin{align}
A & =\left(\begin{array}{ccc}
\frac{1+g_{11}t\left(1+t\right)}{1+t+t^{2}} & \frac{t\left(1-g_{11}\right)}{\left(1+t+t^{2}\right)}+tg_{12} & *\\
\frac{1+g_{21}t\left(1+t\right)}{1+t+t^{2}} & \frac{t\left(1-g_{21}\right)}{\left(1+t+t^{2}\right)}+tg_{22} & *\\
* & * & *
\end{array}\right).
\end{align}

\noindent \begin{definition}

Define the \emph{first homomorphism} $\phi:\mathcal{B}\to\mathrm{GL}\left(2,R\left[t,t^{-1},\left(1+t\right)^{-1}\right]\right)$
by
\begin{align*}
A & \mapsto\left(\begin{array}{cc}
g_{11} & g_{12}\\
t^{-1}g_{11}+\left(1+t\right)g_{21} & t^{-1}g_{12}+\left(1+t\right)g_{22}
\end{array}\right).
\end{align*}
\end{definition}
\begin{lemma}

The map $\phi$ is an injective group homomorphism preserving the
determinant.

\noindent \end{lemma}
\begin{proof}

\noindent The ``homomorphism'' part follows by computing directly
for any pair $A_{1},\,A_{2}\in\mathcal{B}$,
\begin{align*}
\phi\left(A_{1}\right)\phi\left(A_{2}\right) & =\phi\left(A_{1}A_{2}\right).
\end{align*}

By definition, the determinant of the image
\begin{align*}
 & \,\det\phi\left(A\right)\\
= & \,g_{11}\left(t^{-1}g_{12}+\left(1+t\right)g_{22}\right)-g_{12}\left(t^{-1}g_{11}+\left(1+t\right)g_{21}\right)\\
= & \,\left(1+t\right)\left(g_{11}g_{22}-g_{12}g_{21}\right),
\end{align*}
where we conclude that $\det\left(A\right)=\det\phi\left(A\right)$
by (8).

Finally, we show the injectivity. For a matrix $A\in\mathcal{B}$,
suppose
\begin{align*}
\phi\left(A\right) & =\left(\begin{array}{cc}
g_{11} & g_{12}\\
t^{-1}g_{11}+\left(1+t\right)g_{21} & t^{-1}g_{12}+\left(1+t\right)g_{22}
\end{array}\right)=\left(\begin{array}{cc}
1 & 0\\
0 & 1
\end{array}\right).
\end{align*}

By (7) and the definition of $g_{ij}$, we have
\begin{align*}
 & t\left(1+t\right)=f_{11}=A_{11}\left(1+t+t^{2}\right)-1\Rightarrow A_{11}=1,\\
 & 0=f_{12}=A_{11}+A_{12}\left(1+t\right)-1\Rightarrow A_{12}=0,\\
 & 0=t^{-1}\left(1+f_{21}\right)=t^{-1}\left(A_{21}\left(1+t+t^{2}\right)\right)\Rightarrow A_{21}=0,\\
 & 1=t^{-1}f_{22}=t^{-1}\left(A_{21}+A_{22}\left(1+t\right)-1\right)\Rightarrow A_{22}=1,
\end{align*}
and since the other entries of $A$ are determined by these four entries,
we conclude that $A$ is the identity matrix. $\qedhere$

\noindent \end{proof}

We now express $f_{21}$ and $f_{22}$ in terms of $f_{11}$, $f_{12}$
and the determinant, by using the functional equations in Lemma 3.2.

\noindent \begin{lemma}

For a matrix $A\in\mathcal{B}$, when we write the entries of $A$
in terms of $\left\{ f_{kl}\right\} _{1\le k,\,l\le2}$ as in (11),
we have
\begin{align}
 & f_{21}=-\left(\frac{f_{11}+\overline{f_{12}}t^{3}\left(1+t+t^{2}\right)\det\left(A\right)}{t\left(1+t\right)}\right),\\
 & f_{22}=\frac{-f_{12}+\overline{f_{11}}t^{4}\det\left(A\right)}{t\left(1+t\right)}.
\end{align}
\end{lemma}
\begin{proof}

\noindent Suppose $f_{12}\ne0$. The equation (9) implies
\begin{align}
f_{22} & =\frac{-1-t-\overline{f_{11}}f_{21}t^{2}}{\overline{f_{12}}t\left(1+t+t^{2}\right)}.
\end{align}

On the other hand, the equation (8) implies
\begin{align}
f_{22} & =\frac{f_{12}f_{21}+t^{2}\left(1+t\right)\det\left(A\right)}{f_{11}},
\end{align}
where $f_{11}\ne0$. Otherwise, from (9) we have $\left(1+t+t^{-1}\right)f_{12}\overline{f_{12}}=\frac{\left(1+t\right)^{2}}{t}$,
which is impossible for a Laurent polynomial $f_{12}\in\mathbb{Z}\left[t,t^{-1}\right]$.
Equating the right-hand sides of (15) and (16), we obtain (13):
\begin{align*}
f_{21} & =-\left(\frac{f_{11}+\overline{f_{12}}t^{3}\left(1+t+t^{2}\right)\det\left(A\right)}{t\left(1+t\right)}\right),
\end{align*}
where the denominator is simplified by (9). Putting (13) into the
equation (10), we have (14):
\begin{align*}
f_{22} & =\frac{-f_{12}+\overline{f_{11}}t^{4}\det\left(A\right)}{t\left(1+t\right)},
\end{align*}
where the numerator is simplified by (9).

On the other hand, suppose $f_{12}=0$. From (10) and (8), we have
\begin{align}
f_{21} & =\frac{-1-t}{\overline{f_{11}}t^{2}},\;f_{22}=\frac{t^{2}\left(1+t\right)\det\left(A\right)}{f_{11}}.
\end{align}

By substituting (17) into (9), we have a functional equation
\begin{align*}
f_{11}\overline{f_{11}} & =\left(1+t\right)\left(1+t^{-1}\right),
\end{align*}
which implies $f_{11}=\pm\left(1+t\right)t^{m}$ for some integer
$m$. By using these solutions, we have (13) and (14) directly from
(17). $\qedhere$

\noindent \end{proof}
\begin{corollary}

The first homomorphism $\phi$ on $\mathcal{B}$ maps $A\in\mathcal{B}$
into
\begin{align*}
\left(\begin{array}{cc}
g_{11} & g_{12}\\
-\det\left(A\right)\left(1+t^{-1}+t\right)\overline{g_{12}} & \det\left(A\right)\overline{g_{11}}
\end{array}\right).
\end{align*}
\end{corollary}
\begin{proof}

\noindent In terms of $\left\{ g_{kl}\right\} _{1\le k,\,l\le2}$,
the equations (13) and (14) in Lemma 3.5 are written as
\begin{align*}
 & g_{21}=-\left(\frac{g_{11}+\overline{g_{12}}\left(1+t+t^{2}\right)\det\left(A\right)}{t\left(1+t\right)}\right),\\
 & g_{22}=\frac{-g_{12}+\overline{g_{11}}t\det\left(A\right)}{t\left(1+t\right)}.
\end{align*}
which simplifies the image of the map $\phi$ as desired. $\qedhere$

\noindent \end{proof}

On the other hand, for $A\in\mathcal{B}$, we can evalulate the entries
at $t\mapsto-1$. It defines the evaluation map:
\begin{align*}
\mathrm{eval}_{-1} & :\mathcal{B}\to\mathrm{GL}\left(3,\,\mathbb{Z}\right).
\end{align*}

For a matrix $B\in\mathrm{eval}_{-1}\left(\mathcal{B}\right)$, we
write the entries in terms of the four corner entries by (5) as
\begin{align}
B & =\left(\begin{array}{ccc}
B_{11} & 1-B_{11}-B_{13} & B_{13}\\
-1+B_{11}+B_{31} & 3-B_{11}-B_{13}-B_{31}-B_{33} & -1+B_{13}+B_{33}\\
B_{31} & 1-B_{31}-B_{33} & B_{33}
\end{array}\right).
\end{align}

\noindent \begin{definition}

Define the \emph{second homomorphism} $\rho:\mathrm{eval}_{-1}\left(\mathcal{B}\right)\to\mathrm{GL}\left(2,\,\mathbb{Z}\right)$
by
\begin{align*}
B & \mapsto\left(\begin{array}{cc}
1-B_{13} & 1-B_{11}\\
1-B_{33} & 1-B_{31}
\end{array}\right).
\end{align*}
\end{definition}
\begin{lemma}

The map $\rho$ is an injective group homomorphism preserving the
determinant.

\noindent \end{lemma}
\begin{proof}

\noindent It is a homomorphism by computing directly for any pair
$B_{1},\,B_{2}\in\mathrm{eval}_{-1}\left(\mathcal{B}\right)$,
\begin{align*}
\rho\left(B_{1}\right)\rho\left(B_{2}\right) & =\rho\left(B_{1}B_{2}\right).
\end{align*}

On the determinant, from (18) we have
\begin{align*}
\det B & =B_{11}-B_{13}-B_{31}+B_{33}+B_{13}B_{31}-B_{11}B_{33},
\end{align*}
which is equal to $\det\rho\left(B\right)=\det\left(\begin{array}{cc}
1-B_{13} & 1-B_{11}\\
1-B_{33} & 1-B_{31}
\end{array}\right)$.

We show the injectivity. For a matrix $B\in\mathrm{eval}_{-1}\left(\mathcal{B}\right)$,
suppose
\begin{align*}
\rho\left(B\right) & =\left(\begin{array}{cc}
1-B_{13} & 1-B_{11}\\
1-B_{33} & 1-B_{31}
\end{array}\right)=\left(\begin{array}{cc}
1 & 0\\
0 & 1
\end{array}\right).
\end{align*}

Then, we have $\left(\begin{array}{cc}
B_{11} & B_{13}\\
B_{31} & B_{33}
\end{array}\right)=\left(\begin{array}{cc}
1 & 0\\
0 & 1
\end{array}\right)$ from the equations at the entries, and conclude that $B$ is the
identity matrix. $\qedhere$

\noindent \end{proof}

We compute the images of $\beta\left(\sigma_{1}\right)=\left(\begin{array}{ccc}
1-t & t & 0\\
1 & 0 & 0\\
0 & 0 & 1
\end{array}\right)$ and $\beta\left(\sigma_{2}\right)=\left(\begin{array}{ccc}
1 & 0 & 0\\
0 & 1-t & t\\
0 & 1 & 0
\end{array}\right)$ under the two homomorphisms $\phi$ and $\rho$. They are given by
\begin{align}
 & \phi\left(\beta\left(\sigma_{1}\right)\right)=\left(\begin{array}{cc}
-\frac{t^{2}}{1+t} & \frac{t}{1+t}\\
\frac{1+t+t^{2}}{1+t} & \frac{1}{1+t}
\end{array}\right),\;\phi\left(\beta\left(\sigma_{2}\right)\right)=\left(\begin{array}{cc}
1 & 0\\
0 & -t
\end{array}\right),\\
 & \rho\left(\beta\left(\sigma_{1}\right)|_{t=-1}\right)=\left(\begin{array}{cc}
1 & -1\\
0 & 1
\end{array}\right),\;\rho\left(\beta\left(\sigma_{2}\right)|_{t=-1}\right)=\left(\begin{array}{cc}
1 & 0\\
1 & 1
\end{array}\right).
\end{align}

It is natural to use the \emph{projective} linear groups in further
calculations. Put $\pi$ as the natural projection
\begin{align*}
\pi & :\mathrm{GL}\left(2,\,\mathbb{Q}\left(t\right)\right)\to\mathrm{PGL}\left(2,\,\mathbb{Q}\left(t\right)\right),
\end{align*}
where $\ker\pi$ includes the scalar matrices $uI$, for units $u$
in $\mathbb{Q}\left(t\right)$. Denote by $\overline{\phi}$ (resp.
$\overline{\rho}$) the composition $\pi\circ\phi$ (resp. $\pi\circ\rho\circ\mathrm{eval}_{-1}$). 

From (20), we see the restriction $\overline{\rho}$ on $\mathrm{Bu}$
is surjective onto $\mathrm{PSL}\left(2,\,\mathbb{Z}\right)$, since
the images of the two integral matrices generate the whole modular
group.

\noindent \begin{lemma}

The kernel of $\overline{\rho}|_{\mathrm{Bu}}:\mathrm{Bu}\to\mathrm{PSL}\left(2,\mathbb{Z}\right)$
is the center $Z\left(\mathrm{Bu}\right)$, generated by $\beta\left(\Delta\right)$.
In other words, $\overline{\rho}$ induces the short exact sequence:
\begin{align*}
1 & \longrightarrow\mathbb{Z}\longrightarrow\mathrm{Bu}\longrightarrow\mathrm{PSL}\left(2,\mathbb{Z}\right)\longrightarrow1.
\end{align*}

\noindent \end{lemma}
\begin{proof}

\noindent Recall (4), where the Burau image of $\Delta=\left(\sigma_{1}\sigma_{2}\sigma_{1}\right)^{2}$
is given by
\begin{align*}
\beta\left(\Delta\right) & =\left(\begin{array}{ccc}
1-t+t^{3} & t-t^{2} & t^{2}-t^{3}\\
1-t & t-t^{2}+t^{3} & t^{2}-t^{3}\\
1-t & t-t^{2} & t^{2}
\end{array}\right),
\end{align*}

\noindent where by taking the evaluation at $t=-1$, we have
\begin{align*}
\mathrm{eval}_{-1}\left(\beta\left(\Delta\right)\right) & =\left(\begin{array}{ccc}
1 & -2 & 2\\
2 & -3 & 2\\
2 & -2 & 1
\end{array}\right).
\end{align*}

Therefore,

\[
\overline{\rho}\left(\beta\left(\Delta\right)\right)=\left[\begin{array}{cc}
-1 & 0\\
0 & -1
\end{array}\right]=\left[\begin{array}{cc}
1 & 0\\
0 & 1
\end{array}\right],
\]

\noindent which implies $\beta\left(\Delta\right)\in\ker\overline{\rho}$.
From the usual short exact sequence on the braid group $B_{3}$ (for
example, see \citep[p. 22]{MR2435235})
\begin{align*}
1 & \longrightarrow\mathbb{Z}\longrightarrow B_{3}\longrightarrow\mathrm{PSL}\left(2,\mathbb{Z}\right)\longrightarrow1,
\end{align*}
we conclude the proof by the four lemma. $\qedhere$

\noindent \end{proof}

How about the image of $\overline{\rho}$ of the whole group $\mathcal{B}$?
The answer is the same as that of $\mathrm{Bu}$.

\noindent \begin{lemma}

For a matrix $A\in\mathcal{B}$, we have $\det\left(A\right)=\left(-t\right)^{k}$
for some integer $k$. In particular, we have $\overline{\rho}\left(\mathcal{B}\right)=\mathrm{PSL}\left(2,\,\mathbb{Z}\right)$.

\noindent \end{lemma}
\begin{proof}

\noindent At the end of Section 2, we mentioned the determinant must
be $\pm t^{k}$ for some integer $k$. Thus, it suffices to show that
$\det\left(A\right)|_{t=-1}=1$. At first, from (7) we have
\begin{align}
f_{11}|_{t=-1} & =A_{11}|_{t=-1}-1=f_{12}|_{t=-1}.
\end{align}

Recall the equation (14) in Lemma 3.5 as

\[
f_{22}=\frac{-f_{12}+\overline{f_{11}}t^{4}\det\left(A\right)}{t\left(1+t\right)},
\]

\noindent where since $f_{22}$ is a Laurent polynomial, the numerator
of the right-hand side has a factor $\left(1+t\right)$. In other
word, we have
\begin{align}
\left(f_{11}|_{t=-1}\right)\left(\det\left(A\right)|_{t=-1}\right) & =f_{12}|_{t=-1}.
\end{align}

Comparing (21) with (22), when $f_{11}|_{t=-1}\ne0$, the fact that
$\det\left(A\right)|_{t=-1}=1$ follows, which is required. Suppose
now $f_{11}|_{t=-1}=0=f_{12}|_{t=-1}.$ Then, $g_{11}=\frac{f_{11}}{t\left(1+t\right)}$
and $g_{12}=\frac{f_{12}}{t\left(1+t\right)}$ are also Laurent polynomials
with integer coefficients, satisfying a functional equation:
\begin{align}
g_{11}\overline{g_{11}}+\left(1+t^{-1}+t\right)g_{12}\overline{g_{12}} & =1,
\end{align}

\noindent which is direct from (9).

On the other hand, represent the matrix entry $A_{22}$ in terms of
$g_{11},\:g_{12}$ as
\begin{align*}
A_{22} & =\frac{t\left(1+t\right)+g_{11}+\left(\det\left(A\right)\left(t\overline{g_{11}}+\overline{g_{12}}\right)-g_{12}\right)\left(1+t+t^{2}\right)}{\left(1+t\right)\left(1+t+t^{2}\right)}.
\end{align*}

Since $A_{22}$ is a Laurent polynomial, the numerator satisfies
\begin{align*}
\left(-1+\det\left(A\right)|_{t=-1}\right)\left(-g_{11}|_{t=-1}+g_{12}|_{t=-1}\right) & =0,
\end{align*}
which implies $\det\left(A\right)|_{t=-1}=1$ or $g_{11}|_{t=-1}=g_{12}|_{t=-1}$.
However, the latter is always impossible from the evaluation of (23)
at $t=-1$. $\qedhere$

\noindent \end{proof}

The following lemma provides a characterization of the image of $\phi$.

\noindent \begin{lemma}

Suppose a matrix $M=\left(\begin{array}{cc}
g_{11} & g_{12}\\
-\det\left(M\right)\left(1+t^{-1}+t\right)\overline{g_{12}} & \det\left(M\right)\overline{g_{11}}
\end{array}\right)$ which is included in $\mathrm{GL}\left(2,\,\mathbb{Z}\left[t,t^{-1}\right]\right)$
satisfies that $\det\left(M\right)=\left(-t\right)^{k}$ for some
integer $k$. Then, the image group $\phi\left(\mathcal{B}\right)$
includes $M$ if and only if $g_{11}|_{t=\zeta_{3}}=1$, where $\zeta_{3}$
is a third root of unity satisfying the equation $\zeta_{3}^{2}+\zeta_{3}+1=0.$

\noindent \end{lemma}
\begin{proof}

\noindent Let us assume the premise of the lemma. Then, $M\in\phi\left(\mathcal{B}\right)$
if and only if there exists a matrix $A\in\mathcal{B}$ such that
$M=\phi\left(A\right)$. From (12), we have
\begin{align*}
 & A_{11}=\frac{1+g_{11}t\left(1+t\right)}{1+t+t^{2}},\;A_{12}=\frac{t\left(1-g_{11}\right)}{\left(1+t+t^{2}\right)}+tg_{12},\\
 & A_{21}=\frac{1+g_{21}t\left(1+t\right)}{1+t+t^{2}},\;A_{22}=\frac{t\left(1-g_{21}\right)}{\left(1+t+t^{2}\right)}+tg_{22},
\end{align*}
which implies that $A_{11}$, $A_{12}$, $A_{21}$ and $A_{22}$ are
Laurent polynomials if and only if
\begin{align*}
g_{11}|_{t=\zeta_{3}}\left(\zeta_{3}+\zeta_{3}^{2}\right) & =-1,\:g_{21}|_{t=\zeta_{3}}\left(\zeta_{3}+\zeta_{3}^{2}\right)=-1,
\end{align*}
if and only if $g_{11}|_{t=\zeta_{3}}=1=g_{21}|_{t=\zeta_{3}}$. From
(13), it is direct that $g_{11}|_{t=\zeta_{3}}=1$ if and only if
$g_{21}|_{t=\zeta_{3}}=1$, which simplifies the condition. $\qedhere$

\noindent \end{proof}
\begin{corollary}

A matrix $A\in\mathcal{B}$ is contained in $\ker\overline{\phi}$
if and only if $A=\beta\left(\Delta\right)^{k}$ for some $k\in\mathbb{Z}$.

\noindent \end{corollary}
\begin{proof}

\noindent Recall (4), where the Burau image of $\Delta=\left(\sigma_{1}\sigma_{2}\sigma_{1}\right)^{2}$
is given by
\begin{align*}
\beta\left(\Delta\right) & =\left(\begin{array}{ccc}
1-t+t^{3} & t-t^{2} & t^{2}-t^{3}\\
1-t & t-t^{2}+t^{3} & t^{2}-t^{3}\\
1-t & t-t^{2} & t^{2}
\end{array}\right),
\end{align*}
by which we compute
\begin{align*}
\phi\left(\beta\left(\Delta\right)\right) & =\left(\begin{array}{cc}
t^{3} & 0\\
0 & t^{3}
\end{array}\right),
\end{align*}

\noindent which is included in $\ker\overline{\phi}$.

On the other hand, choose a matrix $A\in\mathcal{B}$. Suppose that
$A\in\ker\overline{\phi}$, which is equivalent to the condition that
$\phi\left(A\right)=uI$ for some unit $u=\pm t^{l_{1}}\left(1+t\right)^{l_{2}}$
in $\mathbb{Z}\left[t,\,t^{-1},\,\left(1+t\right)^{-1}\right]^{\times}$,
where $l_{1}$ and $l_{2}$ are integers. By Corollary 3.6, we have
\begin{align}
\phi\left(A\right) & =\left(\begin{array}{cc}
g_{11} & 0\\
0 & \det\left(A\right)\overline{g_{11}}
\end{array}\right).
\end{align}

Since the determinant is preserved under $\phi$ by Lemma 3.4, by
equating $\det\phi\left(A\right)$ with $\det A$ in (24), we have
a functional equation
\begin{align*}
g_{11}\overline{g_{11}} & =1,
\end{align*}
which has only solutions $g_{11}=\pm t^{l}$ for some integer $l$
in $g_{11}\in\mathbb{Z}\left[t,\,t^{-1},\,\left(1+t\right)^{-1}\right]$.
From Lemma 3.11, we need to require that
\begin{align*}
1 & =g_{11}|_{t=\zeta_{3}}=\pm\zeta_{3}^{l},
\end{align*}
which implies $g_{11}=\left(t^{3}\right)^{k}$ for some integer $k$.
$\qedhere$

\noindent \end{proof}
\begin{corollary}

The center of the formal Burau group $\mathcal{B}$ is generated by
$\beta\left(\Delta\right)$. In particular, we have $Z\left(\mathcal{B}\right)=Z\left(\mathrm{Bu}\right)$.

\noindent \end{corollary}
\begin{proof}

\noindent Suppose a matrix $A\in\mathcal{B}$ is in the center $Z\left(\mathcal{B}\right)$
and define $\left\{ g_{ij}\right\} _{1\le i,\,j\le2}$ as usual on
$A$. Since $A$ commutes with $\beta\left(\sigma_{2}\right)$, we
have
\begin{align}
\phi\left(A\right)\phi\left(\beta\left(\sigma_{2}\right)\right) & =\phi\left(\beta\left(\sigma_{2}\right)\right)\phi\left(A\right).
\end{align}

We expand (25) from (19) as
\begin{align*}
\left(\begin{array}{cc}
g_{11} & -g_{12}t\\
* & *
\end{array}\right) & =\left(\begin{array}{cc}
g_{11} & g_{12}\\
* & *
\end{array}\right),
\end{align*}
from which we have $g_{12}=0$. As in the proof of Corollary 3.12,
Lemma 3.11 enforces ${g_{11}=\left(t^{3}\right)^{k}}$ for some integer
$k$. From Lemma 3.10, the only possible candidate for $A$ is of
form $\beta\left(\Delta^{m}\sigma_{2}^{n}\right)$ for a pair of integers
$m,\:n$. By using (19) again, we conclude that $n$ must be 0. $\qedhere$

\noindent \end{proof}

By Lemma 3.9 and Corollary 3.12, we have
\begin{align*}
\ker\overline{\rho}|_{\mathrm{Bu}} & =\left\langle \beta\left(\Delta\right)\right\rangle =\ker\overline{\phi},
\end{align*}
from which we can reconstruct the map $\mu$ taking the cosets $\left\langle \beta\left(\Delta\right)\right\rangle A$
in $\mathrm{Bu}$ to $\overline{\phi}\left(A\right)$, from the matrices
in (19) and (20). A remarkable property of $\mu$ is its injectivity,
which ultimately depends on the faithfulness of the Burau representation
$\beta$.

\noindent \begin{lemma}

The map $\mu$ is injective.

\noindent \end{lemma}
\begin{proof}

\noindent By Lemma 3.9 and Corollary 3.12, we have a commutative diagram:
\begin{align*}
\xymatrix{1\ar[r] & \mathbb{Z}\ar[r]\ar[d]^{=} & \mathrm{Bu}\ar[d]^{i}\ar[r]^{\overline{\rho}} & \mathrm{PSL}\left(2,\mathbb{Z}\right)\ar[d]^{\mu}\ar[r] & 1\\
1\ar[r] & \mathbb{Z}\ar[r] & \mathcal{B}\ar[r]^{\overline{\phi}} & \overline{\phi}\left(\mathcal{B}\right)\ar[r] & 1
}
\end{align*}
where $i:\mathrm{Bu}\to\mathcal{B}$ is the inclusion map defined
by $i\left(A\right)=A$, and the kernels are generated by $\beta\left(\Delta\right)$.
The four lemma establishes the result desired. $\qedhere$

\noindent \end{proof}

As predicted in Section 2, we prove the following lemma. This justifies
that it is sufficient for us to simply use the formal Burau group
instead of Salter's target group in our analysis.

\noindent \begin{lemma}

The formal Burau group $\mathcal{B}$ is equal to Salter's target
group $\Gamma$.

\noindent \end{lemma}
\begin{proof}

\noindent We show that for each $A\in\mathcal{B}$, the image of $S_{3}$
under the permutation representation includes $A|_{t=1}$. By evaluating
the functional equation (9) at $t=1$, we have a Diophantine equation
\begin{align}
\left(f_{11}|_{t=1}\right)^{2}+3\left(f_{12}|_{t=1}\right)^{2} & =4.
\end{align}

The only integer solutions $\left(f_{11}|_{t=1},\,f_{12}|_{t=1}\right)$
for (26) are $\left(\pm2,\,0\right)$ and $\left(\pm1,\,\pm1\right)$.
From (7), (13), and (14), by applying $\det A|_{t=1}=\pm1$, it is
direct that there are only 6 possible combinations of pairs of integers
\begin{align*}
\left(A_{11}|_{t=1},\,A_{12}|_{t=1},\,A_{21}|_{t=1},\,A_{22}|_{t=1}\right),
\end{align*}
exactly corresponding to the image matrices of the permutation representation
of $S_{3}$. $\qedhere$

\noindent \end{proof}

The following lemma, contrary to Lemma 3.11, characterizes the condition
for the image $\phi\left(A\right)$ of an element $A\in\mathcal{B}$
to have Laurent polynomial entries.

\noindent \begin{lemma}

For a matrix $A\in\mathcal{B}$, the matrix $\phi\left(A\right)$
is contained in $\mathrm{GL}\left(2,\,\mathbb{Z}\left[t,t^{-1}\right]\right)$
if and only if $\rho\left(A\right)_{12}=0$.

\noindent \end{lemma}
\begin{proof}

\noindent For a matrix $A\in\mathcal{B}$, Corollary 3.6 ensures that
$\phi\left(A\right)\in\mathrm{GL}\left(2,\,\mathbb{Z}\left[t,t^{-1}\right]\right)$
if and only if $g_{11}$ and $g_{12}$ are Laurent polynomials. By
the definition, $g_{11}$ and $g_{12}$ are Laurent polynomials if
and only if $f_{11}|_{t=-1}=0=f_{12}|_{t=-1}$. From (7), we see the
last is equivalent to the condition that $A_{11}|_{t=-1}=1$. $\qedhere$

\noindent \end{proof}

We are ready to describe the structure of $\overline{\Gamma}$. Define
a map $\epsilon:\overline{\phi}\left(\mathcal{B}\right)\to\mathrm{PSL}\left(2,\,\mathbb{Z}\right)$
by
\begin{align*}
\left\langle \beta\left(\Delta\right)\right\rangle \phi\left(A\right) & \mapsto\overline{\rho}\left(A\right),
\end{align*}
where $\left\langle \beta\left(\Delta\right)\right\rangle \phi\left(A\right)$
is the coset having a representative $\phi\left(A\right)\in\phi\left(\mathcal{B}\right)$.
This is well-defined since $\phi$ is injective, from Lemma 3.4, and
$\overline{\rho}\left(\beta\left(\Delta\right)\right)=1$, which we
saw in the proof of ${\mathrm{Lemma}\;3.9}$.

\noindent \begin{theorem}

The group $\overline{\Gamma}$ is isomorphic to a semidirect product
$F\rtimes\mathrm{PSL}\left(2,\,\mathbb{Z}\right)$, where $F$ is
a subgroup of $\overline{\phi}\left(\mathcal{B}\right)\cap\mathrm{PGL}\left(2,\,\mathbb{Z}\left[t,\,t^{-1}\right]\right)$.

\noindent \end{theorem}
\begin{proof}

\noindent At first, note that $\overline{\Gamma}\cong\overline{\phi}\left(\mathcal{B}\right)$
by Corollary 3.13 and Lemma 3.15. By the construction, the map $\epsilon$
induces a short exact sequence:

\noindent 
\begin{align*}
1 & \longrightarrow\ker\epsilon\longrightarrow\overline{\phi}\left(\mathcal{B}\right)\longrightarrow\mathrm{PSL}\left(2,\,\mathbb{Z}\right)\longrightarrow1.
\end{align*}

Moreover, the map $\mu:\mathrm{PSL}\left(2,\,\mathbb{Z}\right)\to\overline{\phi}\left(\mathcal{B}\right)$
makes this sequence split since it is injective by Lemma 3.14. Thus,
we have a semidirect product
\begin{align*}
\overline{\phi}\left(\mathcal{B}\right) & =\ker\epsilon\rtimes\overline{\phi}\left(\mathrm{Bu}\right),
\end{align*}
where we have $\overline{\phi}\left(\mathrm{Bu}\right)\cong\mathrm{PSL}\left(2,\,\mathbb{Z}\right)$.
Take the group $F$ as $\ker\epsilon$, and now it suffices to show
that $\mathrm{PGL}\left(2,\,\mathbb{Z}\left[t,\,t^{-1}\right]\right)$
contains $\ker\epsilon$. Suppose a matrix $\phi\left(A\right)\in\phi\left(\mathcal{B}\right)$
is a representative of $\overline{\phi}\left(A\right)$, which is
included in $\ker\epsilon$. Then, we have $\rho\left(A\right)_{12}=0$
since the only representatives of the identity in $\mathrm{PSL}\left(2,\,\mathbb{Z}\right)$
are exactly of form $\pm I$ in $\mathrm{GL}\left(2,\,\mathbb{Z}\right)$.
By Lemma 3.16, we conclude that $\mathrm{PGL}\left(2,\,\mathbb{Z}\left[t,\,t^{-1}\right]\right)$
includes $\overline{\phi}\left(A\right)$. $\qedhere$

\noindent \end{proof}

From now on, we call the group $F$ in Theorem 3.17 the \emph{deviating
subgroup}.

\section{Quaternionic Group and the Freeness}

In this section, we will undertake some preliminary work to gather
information about the deviating subgroup $F$. This subgroup is crucial
for approaching Salter's question. Ultimately, we will regard this
group as a subgroup of a larger group, the quaternionic group $\mathcal{Q}$
(${\mathrm{Definition}\;4.3}$), and show that $F$ is free (Corollary
4.11) by presenting a structure theorem for the group $\mathcal{Q}$
(Theorem 4.9).

\noindent \begin{lemma}

Question 2.2 is true if and only if the deviating subgroup $F$ is
trivial.

\noindent \end{lemma}
\begin{proof}

\noindent From Corollary 3.13 and Lemma 3.15, we have $\overline{\Gamma}\cong\overline{\phi}\left(\mathcal{B}\right)$.
Under this isomorphism, by the construction, the condition that the
map $\mu:\mathrm{PSL}\left(2,\,\mathbb{Z}\right)\to\overline{\phi}\left(\mathcal{B}\right)$
is surjective is equivalent to Question 2.2 being true. Since this
map makes the sequence induced by $\epsilon$
\begin{align*}
1 & \longrightarrow\ker\epsilon\longrightarrow\overline{\phi}\left(\mathcal{B}\right)\longrightarrow\mathrm{PSL}\left(2,\,\mathbb{Z}\right)\longrightarrow1
\end{align*}

\noindent split, the surjectivity of $\mu$ is again equivalent to
$F$ being trivial, where we have chosen $F=\ker\epsilon$ in the
proof of Theorem 3.17. $\qedhere$

\noindent \end{proof}

From now on, for an element $M\in\mathrm{PGL}\left(2,\,\mathbb{Q}\left[t,\,t^{-1}\right]\right)$
and pairs of integers $1\le i,\,j\le2$, we say $M_{ij}=0$ (resp.
$M_{ij}\ne0$) if $B_{ij}=0$ (resp. $B_{ij}\ne0$), for each representative
$B$ of $M$ in $\mathrm{GL}\left(2,\,\mathbb{Q}\left[t,\,t^{-1}\right]\right)$.

\noindent \begin{lemma}

There is no element ${M\in\overline{\phi}\left(\mathcal{B}\right)\cap\mathrm{PGL}\left(2,\,\mathbb{Z}\left[t,\,t^{-1}\right]\right)}$
such that $M_{12}\ne0$ if and only if the deviating subgroup $F$
is trivial.

\noindent \end{lemma}
\begin{proof}

\noindent Suppose an element $M\in\overline{\phi}\left(\mathcal{B}\right)\cap\mathrm{PGL}\left(2,\,\mathbb{Z}\left[t,\,t^{-1}\right]\right)$
always satisfies $M_{12}=0$. By Corollary 3.6, we have $M_{12}=0$
if and only if $M_{21}=0$. Moreover, for some matrix $A\in\mathcal{B}$
such that $\overline{\phi}\left(A\right)=M$, we have
\[
A=\left(\begin{array}{cc}
g_{11} & 0\\
0 & \det A\overline{g_{11}}
\end{array}\right).
\]

By (9), the condition $g_{11}\overline{g_{11}}=1$ forces the only
possible candidates of $g_{11}$ to be $\pm t^{l}$ for an integer
$l$. By Lemma 3.10, we have $\det A=\left(-t\right)^{k}$ for some
integer $k$. By substituting this, we have
\begin{align*}
M & =\left[\begin{array}{cc}
1 & 0\\
0 & \left(-t\right)^{k-2l}
\end{array}\right],
\end{align*}
which is equal to $\overline{\phi}\left(\beta\left(\sigma_{2}\right)\right)^{k-2l}$
from (19). Therefore, when $M_{12}=0$, we always have $\epsilon\left(M\right)=1$
if and only if $M=1$.

Conversely, suppose there is an element $M\in\overline{\phi}\left(\mathcal{B}\right)\cap\mathrm{PGL}\left(2,\,\mathbb{Z}\left[t,\,t^{-1}\right]\right)$
such that $M_{12}\ne0$. By Lemma 3.16, we have $\epsilon\left(M\right)_{12}=0$.
Since the lower triangular subgroup of $\mathrm{PSL}\left(2,\,\mathbb{Z}\right)$
is generated by
\begin{align*}
\overline{\rho}\left(\beta\left(\sigma_{2}\right)\right) & =\left[\begin{array}{cc}
1 & 0\\
1 & 1
\end{array}\right],
\end{align*}
which is from (20), there exists an integer $m$ such that $\epsilon\left(M\overline{\phi}\left(\beta\left(\sigma_{2}\right)\right)^{m}\right)=1$.
From (19), we have $\overline{\phi}\left(\beta\left(\sigma_{2}\right)\right)^{-m}=\left[\begin{array}{cc}
1 & 0\\
0 & \left(-t\right)^{-m}
\end{array}\right]$, which cannot be the same as $M$. $\qedhere$

\noindent \end{proof}

By Lemma 4.2, it is sufficient for us to analyze the elements of $\overline{\phi}\left(\mathcal{B}\right)\cap\mathrm{PGL}\left(2,\,\mathbb{Z}\left[t,\,t^{-1}\right]\right)$.
For the sake of abbreviation, from now on, denote the Laurent polynomial
$1+t^{-1}+t$ by $\Phi$. Let us define a group, which is a subgroup
with rational determinants of the projective similitude group.

\noindent \begin{definition}

Define the \emph{quaternionic group} as
\begin{align*}
\mathcal{Q}\, & :=\,\left\{ M\in\mathrm{PGL}\left(2,\,\mathbb{Q}\left[t,\,t^{-1}\right]\right)\::\:M=\left[\begin{array}{cc}
g_{1} & g_{2}\\
-\Phi\overline{g_{2}} & \overline{g_{1}}
\end{array}\right],\,\mathrm{where}\;g_{1},\,g_{2}\in\mathbb{Q}\left[t,\,t^{-1}\right]\right\} .
\end{align*}

In addition, define the \emph{integral subgroup} $U$ to be $\mathcal{Q}\cap\mathrm{PGL}\left(2,\,\mathbb{Z}\left[t,\,t^{-1}\right]\right)$,
and define $U^{1}$ to be the subgroup of $U$ of elements having
a representative $B$ such that $\det\left(B\right)=1$.

\noindent \end{definition}

For each element $M\in\mathcal{Q}$, by definition, there exists a
representative $B=\left(\begin{array}{cc}
g_{1} & g_{2}\\
-\Phi\overline{g_{2}} & \overline{g_{1}}
\end{array}\right)$. Since $\overline{\Phi}=\Phi$, we observe that
\begin{align*}
\det\left(B\right) & =\overline{\det\left(B\right)}.
\end{align*}

Since the only palindromic units in $\mathbb{Q}\left[t,\,t^{-1}\right]$
are nonzero rationals, we have a functional equation for elements
of $\mathcal{Q}$ generalizing (23)
\begin{align}
g_{1}\overline{g_{1}}+\left(1+t^{-1}+t\right)g_{2}\overline{g_{2}} & =k,
\end{align}
where $k$ is a nonzero rational. Obviously, the number $k$ depends
on the element $M$ up to multiplication by a square rational.

\noindent \begin{lemma}

The group $\overline{\phi}\left(\mathcal{B}\right)\cap\mathrm{PGL}\left(2,\,\mathbb{Z}\left[t,\,t^{-1}\right]\right)$
contains $U^{1}$.

\noindent \end{lemma}
\begin{proof}

\noindent Take an element $M\in U^{1}$, and choose its representative
$B=\left(\begin{array}{cc}
g_{1} & g_{2}\\
-\Phi\overline{g_{2}} & \overline{g_{1}}
\end{array}\right)$ such that $g_{1},\,g_{2}\in\mathbb{Z}\left[t,\,t^{-1}\right]$ and
$\det\left(B\right)=1$. Then, the condition about the determinant
yields
\begin{align}
g_{1}\overline{g_{1}}+\left(1+t^{-1}+t\right)g_{2}\overline{g_{2}} & =1.
\end{align}

For a third root of unity $\zeta_{3}$, by evaluating (28) at $t=\zeta_{3}$,
we have $\left\Vert g_{1}|_{t=\zeta_{3}}\right\Vert ^{2}=1$, where
$\left\Vert \cdot\right\Vert $ is the usual complex norm. Therefore,
$g_{1}|_{t=\zeta_{3}}$ is a unit in $\mathbb{Z}\left[\zeta_{3}\right]$,
and the only candidates are $\pm1$, $\pm\zeta_{3}$, $\pm\zeta_{3}^{2}$.
Then, exactly one matrix $B'$ among
\begin{align*}
\pm B,\;\pm t^{-1}B,\;\pm t^{-2}B
\end{align*}
satisfies $B'_{11}|_{t=\zeta_{3}}=1$. By Lemma 3.11, we have $B'\in\phi\left(\mathcal{B}\right)$.
Since $B'$ is still a representative of $M$, we also have $M\in\overline{\phi}\left(\mathcal{B}\right)$.$\qedhere$

\noindent \end{proof}

Here is the characterization of the group $U^{1}$ in the image $\overline{\phi}\left(\mathcal{B}\right)$.

\noindent \begin{lemma}

For an element $M\in\overline{\phi}\left(\mathcal{B}\right)\cap\mathrm{PGL}\left(2,\,\mathbb{Z}\left[t,\,t^{-1}\right]\right)$,
the following are equivalent:
\begin{description}
\item [{(a)}] $M\in U^{1}.$
\item [{(b)}] There is a matrix $A\in\mathcal{B}$ such that $M=\overline{\phi}\left(A\right)$
and $\det\left(A\right)=\left(-t\right)^{2l}$ for some integer $l$.
\item [{(c)}] There is a matrix $A\in\mathcal{B}$ such that $M=\overline{\phi}\left(A\right)$
and $\rho\left(A\right)_{21}$ is an even integer.
\end{description}
In particular, the index $\left[\overline{\phi}\left(\mathcal{B}\right)\cap\mathrm{PGL}\left(2,\,\mathbb{Z}\left[t,\,t^{-1}\right]\right)\::\:U^{1}\right]$
is 2, and the coset of $U^{1}$ not including the identity has a representative
$\overline{\phi}\left(\beta\left(\sigma_{2}\right)\right)=\left[\begin{array}{cc}
1 & 0\\
0 & -t
\end{array}\right]$.

\noindent \end{lemma}
\begin{proof}

\noindent $(a)\Rightarrow(b)$ Take an element $M\in U^{1}$, and
choose its representative $B=\left(\begin{array}{cc}
g_{1} & g_{2}\\
-\Phi\overline{g_{2}} & \overline{g_{1}}
\end{array}\right)$ such that $g_{1},\,g_{2}\in\mathbb{Z}\left[t,\,t^{-1}\right]$ and
$\det B=1$. As in the proof of Lemma 4.4, there exists a unique representative
$B'$ of $M$ among 
\begin{align*}
\pm B,\;\pm t^{-1}B,\;\pm t^{-2}B,
\end{align*}

\noindent such that there exists a matirx $A\in\mathcal{B}$ such
that $\phi\left(A\right)=B'$. By the construction, the possible candidates
of $\det A$ are 1, $\left(-t\right)^{-2}$, and $\left(-t\right)^{-4}$,
from the preservation of the determinant (${\mathrm{Lemma}\;3.4}$)
and Corollary 3.6.

\noindent $(b)\Rightarrow(a)$ Suppose $\phi\left(A\right)=\left(\begin{array}{cc}
g_{11} & g_{12}\\
-\Phi\det\left(A\right)\overline{g_{12}} & \det\left(A\right)\overline{g_{11}}
\end{array}\right)$ from Corollary 3.6.

Since $\det A=\left(-t\right)^{2l}$ for some integer $l$, the matrix
$\left(-t\right)^{-l}\phi\left(A\right)$ is of the form
\begin{align*}
\left(\begin{array}{cc}
g_{1} & g_{2}\\
-\Phi\overline{g_{2}} & \overline{g_{1}}
\end{array}\right),
\end{align*}
which implies the group $U^{1}$ includes $M=\overline{\phi}\left(A\right)$. 

\noindent $(b)\Leftrightarrow(c)$ The value $\rho\left(A\right)_{21}$
was defined to be $1-\left(A|_{t=-1}\right)_{33}$. In this proof,
temporarily denote by $f'$ the derivative of the Laurent polynomial
$f$ on the variable $t$, and denote by $f^{e}$ the evaluation $f|_{t=-1}$. 

\noindent \begin{claim1}

For a matrix $A\in\mathcal{B}$ such that $\phi\left(A\right)\in\mathrm{GL}\left(2,\,\mathbb{Z}\left[t,\,t^{-1}\right]\right)$,
we have
\begin{align}
1-A_{33}^{e} & =2\left(g_{11}'\right)^{e}-2\left(g_{12}'\right)^{e}+\left(d'\right)^{e}\left(g_{12}^{e}-g_{11}^{e}\right).
\end{align}
\end{claim1}
\begin{proofclaim1}

\noindent By the assumption, we may assume that $g_{11},\,g_{12}$
are Laurent polynomials and $f_{11}^{e}=0=f_{12}^{e}$. From (12)
and (11), we have
\begin{align}
 & A_{11}^{e},=1,\:A_{12}^{e}=g_{11}^{e}-g_{12}^{e}-1,\\
 & A_{21}^{e}=1+f_{21}^{e},\:A_{22}^{e}=-1-f_{22}^{e}-\left(f_{21}'\right)^{e}+\left(f_{22}'\right)^{e}.
\end{align}

Denote by $d$ the determinant $\det A$. From (13), (14), and the
definition $g_{ij}=\frac{f_{ij}}{t+t^{2}}$, by applying L'Hopital's
rule, we expand

\noindent 
\begin{align}
 & f_{21}^{e}=-g_{11}^{e}-g_{12}^{e},\\
 & f_{22}^{e}=-g_{11}^{e}-g_{12}^{e},\\
 & \left(f_{21}'\right)^{e}=g_{12}^{e}-\left(d'\right)^{e}g_{12}^{e}-\left(g_{11}'\right)^{e}+\left(g_{12}'\right)^{e}\\
 & \left(f_{22}'\right)^{e}=g_{11}^{e}-\left(d'\right)^{e}g_{11}^{e}-\left(g_{12}'\right)^{e}+\left(g_{11}'\right)^{e}
\end{align}

\noindent where we simplified the equations by using $f_{11}^{e}=0=f_{12}^{e}$
and $d^{e}=1$. By substituting (32), (33), (34), and (35) into (31),
we have
\begin{align}
 & A_{21}^{e}=1-g_{11}^{e}-g_{12}^{e},\\
 & A_{22}^{e}=-1+2g_{11}^{e}+2\left(g_{11}'\right)^{e}-2\left(g_{12}'\right)^{e}+\left(d'\right)^{e}\left(g_{12}^{e}-g_{11}^{e}\right).
\end{align}

From (18), we deduce (29) as
\begin{align*}
 & \,1-A_{33}^{e}\\
= & \,A_{13}^{e}-A_{23}^{e}\\
= & \,A_{21}^{e}+A_{22}^{e}-A_{11}^{e}-A_{12}^{e}\\
= & \,2\left(g_{11}'\right)^{e}-2\left(g_{12}'\right)^{e}+\left(d'\right)^{e}\left(g_{12}^{e}-g_{11}^{e}\right),
\end{align*}

\noindent where we used the substitution from (30), (36), and (37)
at the end.$\qed$

\noindent \end{proofclaim1}

We return to the proof of Lemma 4.5. By dividing (9) by $\frac{\left(1+t\right)^{2}}{t}$
and evaluating it at $t=-1$, we have a Diophantine equation
\begin{align*}
\left(g_{11}^{e}-g_{12}^{e}\right)\left(g_{11}^{e}+g_{12}^{e}\right) & =1,
\end{align*}

\noindent which implies $\left(g_{12}^{e}-g_{11}^{e}\right)$ in (29)
is $\pm1$, and especially an odd integer. Put $d=\left(-t\right)^{l}$.
Then we have $\left(d'\right)^{e}=-l$; by (29) in Claim 1, we conclude
the proof. $\qedhere$

\noindent \end{proof}
\begin{corollary}

The group $U^{1}$ contains the deviating subgroup $F$.

\noindent \end{corollary}
\begin{proof}

\noindent By definition, each element $M\in F$ satisfies the condition
(c) in Lemma 4.5. $\qedhere$

\noindent \end{proof}
\begin{corollary}

Question 2.2 is true if and only if there is no element $M\in U^{1}$
such that $M_{12}\ne0$.

\noindent \end{corollary}
\begin{proof}

\noindent By Lemma 4.1 and Lemma 4.2, it suffices to show that there
is an element $M\in U^{1}$ such that $M_{12}\ne0$ if and only if
there is an element $M\in\overline{\phi}\left(\mathcal{B}\right)\cap\mathrm{PGL}\left(2,\,\mathbb{Z}\left[t,\,t^{-1}\right]\right)$
such that $M_{12}\ne0$. The ``only if'' part is trivial by Lemma
4.4.

Suppose there is an element $M\in\overline{\phi}\left(\mathcal{B}\right)\cap\mathrm{PGL}\left(2,\,\mathbb{Z}\left[t,\,t^{-1}\right]\right)$
such that $M_{12}\ne0$. By ${\mathrm{Lemma}\;4.5}$, we have $M\in U^{1}$
or $M\left[\begin{array}{cc}
1 & 0\\
0 & \left(-t\right)^{-1}
\end{array}\right]\in U^{1}$. For each case, we have $M_{12}\ne0$ or
\begin{align*}
\left(M\left[\begin{array}{cc}
1 & 0\\
0 & \left(-t\right)^{-1}
\end{array}\right]\right)_{12} & \ne0,
\end{align*}
by assumption. $\qedhere$

\noindent \end{proof}

What are candidates for elements of $\mathcal{Q}$ or $U$? While
a satisfactory answer regarding $U$ should be postponed to Section
5, we have an immediate answer regarding $\mathcal{Q}$. For a rational
number $r\in\mathbb{Q}$, note that
\begin{align*}
 & \,\det\left(\begin{array}{cc}
t-r^{2} & r\\
-r\left(1+t^{-1}+t\right) & t^{-1}-r^{2}
\end{array}\right)\\
= & \,1-r^{2}\left(t^{-1}+t\right)+r^{4}+r^{2}\left(1+t^{-1}+t\right)\\
= & \,1+r^{2}+r^{4},
\end{align*}
which is always a positive rational number.

\noindent \begin{definition}

For a rational number $r\in\mathbb{Q}$, define a matrix $g\left[r\right]$
as

\begin{equation}
g\left[r\right]\,:=\,\left(\begin{array}{cc}
t-r^{2} & r\\
-r\left(1+t^{-1}+t\right) & t^{-1}-r
\end{array}\right),
\end{equation}

\noindent which is included in $\mathrm{GL}\left(2,\,\mathbb{Q}\left[t,\,t^{-1}\right]\right)$.
Define the \emph{elementary generator} $\overline{g}\left[r\right]$
to be its projectivization $\pi\left(g\left[r\right]\right)$.

\noindent \end{definition}

We are ready to state the main theorem of this section.

\noindent \begin{theorem}

The quaternionic group $\mathcal{Q}$ is freely generated by the set
of all elementary generators $\left\{ \overline{g}\left[r\right]\right\} _{r\in\mathbb{Q}}$.

\noindent \end{theorem}

Before proving this theorem, we first present two immediate applications.

\noindent \begin{corollary}

The integral subgroup $U$ is equal to the subgroup $U^{1}$.

\noindent \end{corollary}
\begin{proof}

\noindent By Theorem 4.9, every element $M\in U$ has a representative
$B\in\mathrm{GL}\left(2,\,\mathbb{Q}\left[t,\,t^{-1}\right]\right)$
represented by a word in elementary generators up to scalar multiplication.
Each elementary generator $g\left[r\right]$ has the determinant $\det g\left[r\right]=1+r^{2}+r^{4}$,
which is always positive. Therefore, we conclude that $\det B$ is
also positive. $\qedhere$

\noindent \end{proof}

\noindent \begin{corollary}

The integral subgroup $U$ and the deviating subgroup $F$ are free.

\noindent \end{corollary}
\begin{proof}

\noindent Since every subgroup of a free group is again free, the
subgroup $U$ of $\mathcal{Q}$ is free from Theorem 4.9. By Corollary
4.6, the subgroup $F$ of $U$ is also free. $\qedhere$

\noindent \end{proof}

\noindent \begin{definition}

For a Laurent polynomial $g$:

\[
g=a_{0}t^{m}+a_{1}t^{m+1}+\cdots+a_{n-m}t^{n},
\]

\noindent with $n\ge m$ integers and $a_{0}\ne0\ne a_{n-m}$, define
the \emph{relative degree} $\mathrm{rd}\left(g\right):=n-m$.

\noindent \end{definition}

If a matrix $\left(\begin{array}{cc}
g_{1} & g_{2}\\
-\Phi\overline{g_{2}} & \overline{g_{1}}
\end{array}\right)$ is in $\mathrm{GL}\left(2,\,\mathbb{Q}\left[t,t^{-1}\right]\right)$,
from the functional equation (27), it is straightforward that

\[
\mathrm{rd}\left(g_{1}\right)=\mathrm{rd}\left(g_{2}\right)+1.
\]

Thus, when $\left(\begin{array}{cc}
g_{1} & g_{2}\\
-\Phi\overline{g_{2}} & \overline{g_{1}}
\end{array}\right)\in\mathrm{GL}\left(2,\,\mathbb{Q}\left[t,t^{-1}\right]\right)$, we also define

\[
\mathrm{rd}\left(\begin{array}{cc}
g_{1} & g_{2}\\
-\Phi\overline{g_{2}} & \overline{g_{1}}
\end{array}\right):=\mathrm{rd}\left(g_{1}\right).
\]

For a matrix $B=\left(\begin{array}{cc}
g_{1} & g_{2}\\
-\Phi\overline{g_{2}} & \overline{g_{1}}
\end{array}\right)\in\mathrm{GL}\left(2,\,\mathbb{Q}\left[t,t^{-1}\right]\right)$ such that $g_{2}\ne0$, write the first row entries as
\begin{align*}
 & g_{1}=a_{0}t^{m_{1}}+a_{1}t^{m_{1}+1}+\cdots+a_{n_{1}-m_{1}}t^{n_{1}},\\
 & g_{2}=b_{0}t^{m_{2}}+b_{1}t^{m_{2}+1}+\cdots+b_{n_{2}-m_{2}}t^{n_{2}}.
\end{align*}

Then, we also have $n_{1}-m_{1}=n_{2}-m_{2}+1$ from (27).

\noindent \begin{definition}

For a matrix $B=\left(\begin{array}{cc}
g_{1} & g_{2}\\
-\Phi\overline{g_{2}} & \overline{g_{1}}
\end{array}\right)\in\mathrm{GL}\left(2,\,\mathbb{Q}\left[t,t^{-1}\right]\right)$ such that $g_{2}\ne0$, write the first row entries as above. Define
$B$ to be \emph{upper-balanced} if $m_{1}=m_{2}$ and $n_{1}=n_{2}+1$,
and \emph{lower-balanced} if $m_{1}=m_{2}-1$ and $n_{1}=n_{2}$.
Also define $B$ to be \emph{balanced} if $B$ is upper-balanced or
lower-balanced. If $g_{2}=0$ or $B$ is not balanced, we call $B$
\emph{unbalanced}.

\noindent \end{definition}
\begin{lemma}

For a matrix $B=\left(\begin{array}{cc}
g_{1} & g_{2}\\
-\Phi\overline{g_{2}} & \overline{g_{1}}
\end{array}\right)\in\mathrm{GL}\left(2,\,\mathbb{Q}\left[t,t^{-1}\right]\right)$, suppose $g_{2}\ne0$. Then, there exists a unique balanced matrix
$\mathrm{bl}\left(B\right)$ among the set of matrices
\begin{align*}
\left\{ Bg\left[0\right]^{k}\;|\;k\in\mathbb{Z}\right\} .
\end{align*}

Let us call $\mathrm{bl}\left(B\right)$ the \emph{balanced companion
of} $B$.

\noindent \end{lemma}
\begin{proof}

From (38), we have $g\left[0\right]=\left(\begin{array}{cc}
t & 0\\
0 & t^{-1}
\end{array}\right)$. Thus, the number $n_{1}-n_{2}$ increases by 2 multiplying $g\left[0\right]$
on the right of $B$, and decreases by 2 multiplying $g\left[0\right]^{-1}$.
Thus, the matrix $Bg\left[0\right]^{k}$ is balanced if and only if
$k=\left\lfloor \frac{n_{2}-n_{1}+1}{2}\right\rfloor $. $\qedhere$

\noindent \end{proof}
\begin{definition}

For a matrix $B=\left(\begin{array}{cc}
g_{1} & g_{2}\\
-\Phi\overline{g_{2}} & \overline{g_{1}}
\end{array}\right)\in\mathrm{GL}\left(2,\,\mathbb{Q}\left[t,t^{-1}\right]\right)$ with $g_{2}\ne0$, we can define the \emph{type} $\tau\left(B\right)$
to be $-a_{0}b_{0}^{-1}$ when $\mathrm{bl}\left(A\right)$ is upper-balanced,
and $-a_{0}^{-1}b_{0}$ when $\mathrm{bl}\left(B\right)$ is lower-balanced.
For an unbalanced matrix $B$, define $\tau\left(B\right)=0$. Abusing
notation, for an element $M=\left[\begin{array}{cc}
g_{1} & g_{2}\\
-\Phi\overline{g_{2}} & \overline{g_{1}}
\end{array}\right]\in\mathcal{Q}$, we also define $\tau\left(M\right)$ to be $\tau\left(B\right)$
for a representative $B$ of $M$.

\noindent \end{definition}
\begin{prooftheorem49}

\noindent We begin the proof of Theorem 4.9, which states that
\begin{description}
\item [{(a)}] \noindent the group $\mathcal{Q}$ is generated by the set
of all elementary generators $\left\{ \overline{g}\left[r\right]\right\} _{r\in\mathbb{Q}}$,
and
\item [{(b)}] \noindent the set of all elementary generators $\left\{ \overline{g}\left[r\right]\right\} _{r\in\mathbb{Q}}$
generates a free subgroup of $\mathcal{Q}$.
\end{description}
\noindent Therefore, the proof has two parts: the generation (a) and
the freeness (b).

\noindent \begin{proofofa}

\noindent We use the induction on the relative degree. Take an element
$M\in\mathcal{Q}$ such that $\mathrm{rd}\left(M\right)=1$ and choose
a representative $B$. We may assume that $B$ is balanced by replacing
it by the balanced companion by Lemma 4.14. By normalizing the entry
$B_{11}$ multiplying by some scalar $t^{l_{0}}$, the only possible
candidates for the representative are
\begin{align*}
B & =\left(\begin{array}{cc}
a_{0}+a_{1}t & b_{0}t^{l}\\
\cdots & \cdots
\end{array}\right),
\end{align*}
where the integer $l$ is 0 if $B$ is upper-balnced, and is $1$
if lower-balanced. Suppose $l=0$. From the functional equation (27),
for some nonzero rational $k$, we have
\begin{align*}
k= & \,\left(a_{0}+a_{1}t\right)\left(a_{0}+a_{1}t^{-1}\right)+\left(1+t+t^{-1}\right)b_{0}^{2}\\
= & \,\left(a_{0}a_{1}+b_{0}^{2}\right)\left(t+t^{-1}\right)+a_{0}^{2}+a_{1}^{2}+b_{0}^{2},
\end{align*}

\noindent from which we deduce $a_{0}a_{1}+b_{0}^{2}=0$ by comparing
the coefficients in both sides. Substituting $a_{0}=-b_{0}^{2}a_{1}^{-1}$,
we have
\begin{align*}
a_{1}^{-1}B & =\left(\begin{array}{cc}
t-\left(b_{0}a_{1}^{-1}\right)^{2} & \left(b_{0}a_{1}^{-1}\right)\\
* & *
\end{array}\right).
\end{align*}

Therefore, by (38), we have
\begin{align*}
\left(a_{1}^{-1}B\right)g\left[b_{0}a_{1}^{-1}\right]^{-1} & =\left(\begin{array}{cc}
1 & 0\\
* & *
\end{array}\right).
\end{align*}

As we assume that $M\in\mathcal{Q}$, there are some integer $l_{1}$
such that $M=\overline{g}\left[0\right]^{l_{1}}\overline{g}\left[b_{0}a_{1}^{-1}\right]$.
The case $l=1$ is almost the same, and we conclude that every element
$M\in\mathcal{Q}$ such that $\mathrm{rd}\left(M\right)=1$ is generated
by elementary generators.

For an integer $N\ge1$, suppose that the set $\left\{ \overline{g}\left[r\right]\right\} _{r\in\mathbb{Q}}$
generates every element $M_{0}\in\mathcal{Q}$ such that $\mathrm{rd}\left(M_{0}\right)\le N$,
and suppose $M\in\mathcal{Q}$ such that $\mathrm{rd}\left(M\right)=N+1$.
Choose a representative $B$ of $M$ of the form
\begin{align*}
B & =\left(\begin{array}{cc}
g_{1} & g_{2}\\
-\Phi\overline{g_{2}} & \overline{g_{1}}
\end{array}\right).
\end{align*}

We assume that $B$ is balanced by replacing $B$ with the balanced
companion by Lemma 4.14. Suppose $B$ is upper-balanced. We write
its entries as
\begin{align*}
 & g_{1}=t^{m_{2}}+a_{1}t^{m_{2}+1}+\cdots+a_{n_{2}-m_{2}+1}t^{n_{2}+1},\\
 & g_{2}=b_{0}t^{m_{2}}+b_{1}t^{m_{2}+1}+\cdots+b_{n_{2}-m_{2}}t^{n_{2}}.
\end{align*}

\noindent where we normalized the first coefficient of $g_{1}$ as
1. We expand then
\begin{align*}
 & \,Bg\left[r\right]^{-1}\\
= & \,B\left(\begin{array}{cc}
t^{-1}-r^{2} & -r\\
r\left(1+t+t^{-1}\right) & t-r^{2}
\end{array}\right)\left(\det g\left[r\right]\right)^{-1}\\
= & \,\left(\begin{array}{cc}
\left(1+b_{0}r\right)t^{m_{2}-1}+\cdots+\left(-a_{n_{2}-m_{2}+1}r^{2}+b_{n_{2}-m_{2}}r\right)t^{n_{2}+1} & *\\
\left(-r-b_{0}r^{2}\right)t^{m_{2}}+\cdots+\left(-a_{n_{2}-m_{2}+1}r+b_{n_{2}-m_{2}}\right)t^{n_{2}+1} & *
\end{array}\right)^{T}\left(\det g\left[r\right]\right)^{-1}.
\end{align*}

If we choose $r=-b_{0}^{-1}$, the lowest degree coefficients above
vanish. Moreover, from the functional equation (27) for $B$, we have

\begin{equation}
a_{n_{2}-m_{2}+1}=-b_{0}b_{n_{2}-m_{2}}.
\end{equation}

By (39) and the choice $r=-b_{0}^{-1}$, the degree $t^{n_{2}+1}$
coefficients of $Bg\left[r\right]^{-1}$ above also vanish. It means
$\mathrm{rd}\left(M\overline{g}\left[r\right]^{-1}\right)\le N-1$,
which is generated by elementary generators by the assumption.

On the other hand, suppose $B$ is lower-balanced. we write its entries
as
\begin{align*}
 & g_{1}=t^{m_{2}-1}+a_{1}t^{m_{2}+1}+\cdots+a_{n_{2}-m_{2}+1}t^{n_{2}},\\
 & g_{2}=b_{0}t^{m_{2}}+b_{1}t^{m_{2}+1}+\cdots+b_{n_{2}-m_{2}}t^{n_{2}}.
\end{align*}

Then, we expand
\begin{align*}
 & \,Bg\left[r\right]\\
= & \,B\left(\begin{array}{cc}
t-r^{2} & r\\
-r\left(1+t+t^{-1}\right) & t^{-1}-r^{2}
\end{array}\right)\\
= & \,\left(\begin{array}{cc}
\left(-r^{2}-b_{0}r\right)t^{m_{2}-1}+\cdots+\left(a_{n_{2}-m_{2}+1}-b_{n_{2}-m_{2}}r\right)t^{n_{2}+1} & *\\
\left(r+b_{0}\right)t^{m_{2}-1}+\cdots+\left(a_{n_{2}-m_{2}+1}r-b_{n_{2}-m_{2}}r^{2}\right)t^{n_{2}} & *
\end{array}\right)^{T},
\end{align*}

\noindent where the lowest and the highest degree coefficients of
$Bg\left[r\right]$ vanish like in the upper-balanced case if we choose
$r=-b_{0}$. $\qed$

\noindent \end{proofofa}
\begin{proofofb}

\noindent We use the usual ping-pong lemma. Consider the action of
$\mathcal{Q}$ on itself by the right group multiplication. Choose
the ping-pong sets $X\left(r\right)\subset\mathcal{Q}$ for each rational
$r$ such that
\begin{description}
\item [{(1)}] The set $X\left(0\right)$ is the set of all unbalanced elements
in $\mathcal{Q}$,
\item [{(2)}] For each nonzero rational $r$, $X\left(r\right)$ is the
set of all type $r$ elements in $\mathcal{Q}$.
\end{description}
\noindent \begin{claim2}

For a rational $r$ and an integer $m>0$, $\overline{g}\left[r\right]^{m}$
has a representative
\begin{align*}
B_{r,\,m} & =\left(\begin{array}{cc}
g_{r,\,m,\,1} & g_{r,\,m,\,2}\\
-\Phi\overline{g_{r,\,m,\,2}} & \overline{g_{r,\,m,\,1}}
\end{array}\right)
\end{align*}
such that
\begin{align*}
 & g_{r,\,m,\,1}=-r^{2}t^{-m+1}+\cdots+t^{m},\\
 & g_{r,\,m,\,2}=rt^{-m+1}+\cdots+rt^{m-1},
\end{align*}

\noindent and $\overline{g}\left[r\right]^{-m}$ has a representative
$B_{r,\,-m}=\left(\begin{array}{cc}
g_{r,\,-m,\,1} & g_{r,\,-m,\,2}\\
-\Phi\overline{g_{r,\,-m,\,2}} & \overline{g_{r,\,-m,\,1}}
\end{array}\right)$ such that
\begin{align*}
 & g_{r,\,-m,\,1}=t^{-m}+\cdots-r^{2}t^{m-1},\\
 & g_{r,\,-m,\,2}=-rt^{-m+1}+\cdots-rt^{m-1}.
\end{align*}

\noindent \end{claim2}
\begin{proofclaim2}

\noindent Temporarily denote by $d$ the determinant $\det g\left[r\right]=1+r^{2}+r^{4}$
in this proof. We use the usual method of finding matrix powers. For
every nonnegative integer $n$, define a sequence of polynomials $\chi_{n}\left(x\right)\in\mathbb{Q}\left[x\right]$
inductively by $\chi_{0}\left(x\right):=0$, $\chi_{1}\left(x\right):=$1
and
\begin{align}
\chi_{n+1}\left(x\right) & :=x\chi_{n}\left(x\right)-d\chi_{n-1}\left(x\right),\;n\ge1.
\end{align}

Then, for each integer $m\ge1$, we have a power expression
\begin{align}
g\left[r\right]^{m} & =\left(\begin{array}{cc}
\left(t-r^{2}\right)\chi_{m}\left(x_{0}\right)-d\chi_{m-1}\left(x_{0}\right) & r\chi_{m}\left(x_{0}\right)\\
-r\chi_{m}\left(x_{0}\right)\Phi & \left(t^{-1}-r^{2}\right)\chi_{m}\left(x_{0}\right)-d\chi_{m-1}\left(x_{0}\right)
\end{array}\right),
\end{align}
where $x_{0}=t^{-1}+t-2r^{2}$. One can easily check (41) by induction;
by (40), the polynomial $\chi_{m}\left(x\right)$ has degree $m-1$
for $m\ge1$. From this fact and (41), the lowest degree coefficient
of $\left(g\left[r\right]^{m}\right)_{11}$ (resp. $\left(g\left[r\right]^{m}\right)_{12}$)
is $-r^{2}$ (resp. $r$). The highest degree coefficient of $\left(g\left[r\right]^{m}\right)_{11}$
(resp. $\left(g\left[r\right]^{m}\right)_{12}$) is $1$ (resp. $r$).

The case $\overline{g}\left[r\right]^{-m}$ is done in the same manner,
by taking the representative $\left(dg\left[r\right]^{-1}\right)^{m}$.
$\qed$

\noindent \end{proofclaim2}

\noindent We return to (b) in the proof of Theorem 4.9. In order to
prove (b) by the ping-pong lemma, for two distinct rationals $r_{1},\,r_{2}$,
a nonzero integer $m$, and an element $M\in X\left(r_{1}\right)$,
it suffices to show that
\begin{align}
\tau\left(M\overline{g}\left[r_{2}\right]^{m}\right) & =r_{2}.
\end{align}

If $r_{1}=0$ or $r_{2}=0$, the result is trivial. Suppose $r_{1}\ne0\ne r_{2}$.
There are four cases regarding the balancedness of $M$ and the sign
of $m$; suppose $M$ is upper-balanced and $m>0$, and take the reprsentatives
$B=\left(\begin{array}{cc}
g_{1} & g_{2}\\
-\Phi\overline{g_{2}} & \overline{g_{1}}
\end{array}\right)$ for $M$ such that
\begin{align*}
 & g_{1}=t^{m_{2}}+a_{1}t^{m_{2}+1}+\cdots+a_{n_{2}-m_{2}+1}t^{n_{2}+1},\\
 & g_{2}=-r_{1}^{-1}t^{m_{2}}+b_{1}t^{m_{2}+1}+\cdots+r_{1}a_{n_{2}-m_{2}+1}t^{n_{2}},
\end{align*}

\noindent where the lowest and the highest degree coefficients in
$g_{1}$ and $g_{2}$ are computed by the given types and the functional
equation (27). By Claim 2, we expand
\begin{align*}
 & \,\left(\begin{array}{c}
\left(BB_{r_{2},\,m}\right)_{11}\\
\left(BB_{r_{2},\,m}\right)_{12}
\end{array}\right)\\
= & \,\left(\begin{array}{c}
g_{1}g_{r_{2},\,m,\,1}-\Phi g_{2}\overline{g_{r_{2},\,m,\,2}}\\
g_{1}g_{r_{2},\,m,\,2}+g_{2}\overline{g_{r_{2},\,m,\,1}}
\end{array}\right)\\
= & \,\left(\begin{array}{c}
r_{2}r_{1}^{-1}t^{m_{2}-m}+\cdots+a_{n_{2}-m_{2}+1}t^{n_{2}+m+1}\\
-r_{1}^{-1}t^{m_{2}-m}+\cdots+r_{2}a_{n_{2}-m_{2}+1}t^{n_{2}+m}
\end{array}\right),
\end{align*}

\noindent where we computed the lowest and the highest degree terms.
Thus, $BB_{r_{2},\,m}$ is upper-balanced and we establish (42). On
the other hand, we also expand
\begin{align*}
 & \,\left(\begin{array}{c}
\left(BB_{r_{2},\,-m}\right)_{11}\\
\left(BB_{r_{2},\,-m}\right)_{12}
\end{array}\right)\\
= & \,\left(\begin{array}{c}
g_{1}g_{r_{2},\,-m,\,1}-\Phi g_{2}\overline{g_{r_{2},\,-m,\,2}}\\
g_{1}g_{r_{2},\,-m,\,2}+g_{2}\overline{g_{r_{2},\,-m,\,1}}
\end{array}\right)\\
= & \,\left(\begin{array}{c}
\left(1-r_{1}^{-1}r_{2}\right)t^{m_{2}-m}+\cdots+\left(r_{1}-r_{2}\right)r_{2}a_{n_{2}-m_{2}+1}t^{n_{2}+m}\\
\left(r_{1}^{-1}r_{2}-1\right)r_{2}t^{m_{2}-m+1}+\cdots+\left(r_{1}-r_{2}\right)a_{n_{2}-m_{2}+1}t^{n_{2}+m}
\end{array}\right).
\end{align*}

From the assumption $r_{1}\ne r_{2}$, we conclude $BB_{r_{2},\,-m}$
is lower-balanced and we establish (42). The cases $M$ is lower-balanced
are treated similarly. $\qed$

\noindent \end{proofofb}
\end{prooftheorem49}
\begin{remark}

One can prove Theorem 4.9 not only using the right group action, but
also using the action on the Bruhat-Tits tree. In fact, the group
$\left\langle \mathcal{Q},\,\left[\begin{array}{cc}
1 & 0\\
0 & -t
\end{array}\right]\right\rangle $ is itself free, and the free generators and the inverses correspond
exactly to the vertices that are one edge away from the lattice of
the identity. However, when one considers generalizations to arbitrary
fields, the proof using the right group action is more useful. For
example, one should require the additional condition that the equation
$x^{2}+1=0$ has no solution over the base field when considering
the Bruhat-Tits tree.

\noindent \end{remark}

\section{Algorithmic Proof of Theorem 1.1}

In Section 4, we proved that the quaternionic group $\mathcal{Q}$
is free and identified the free generators. However, this does not
provide us with any meaningful information directly to find a specific
element in the deviating subgroup $F$. In this secion, we present
an algorithm to find a nontrivial element in $F$ using the elementary
generators of $\mathcal{Q}$, and we find a concrete nontrivial element
to prove Theorem 1.1.

\noindent \begin{definition}

For a matrix $B\in\mathrm{M}\left(2,\,\mathbb{Q}\left[x,x^{-1}\right]\right)$,
a \emph{minimal integral representative} of $B$ is a matrix $B'$
in $\mathrm{M}\left(2,\,\mathbb{Z}\left[x,x^{-1}\right]\right)$ such
that
\begin{description}
\item [{(a)}] there exists an integer $m\in\mathbb{Z}$ with $B'=mB$,
and
\item [{(b)}] for any integer $n>1$, we have $n^{-1}B'\notin\mathrm{M}\left(2,\,\mathbb{Z}\left[x,x^{-1}\right]\right)$.
\end{description}
A minimial integral representative is uniquely determined up to sign;
thus, by identifying $B'$ as an equivalence class defined by $B'\sim-B'$,
we can define \emph{the} minimal integral representative $\mathrm{mir}\left(B\right)\in\mathrm{M}\left(2,\,\mathbb{Z}\left[x,x^{-1}\right]\right)/\sim$
for any $B\in\mathrm{M}\left(2,\,\mathbb{Q}\left[x,x^{-1}\right]\right)$.
Define the \emph{minimal integral representative} of $M\in\mathcal{Q}$
to be $\mathrm{mir}\left(B\right)$ for any representative $B$ of
$M$ such that $\det\left(B\right)\in\mathbb{Z}$.

\noindent \end{definition}

We specifically describe how to find the minimal integral representative.
For two integers $a,\,b$ such that $b\ne0$, from (38) we have
\[
\det g\left[\frac{a}{b}\right]=1+\left(\frac{a}{b}\right)^{2}+\left(\frac{a}{b}\right)^{4},
\]
which implies 
\begin{equation}
b^{4}\det g\left[\frac{a}{b}\right]=a^{4}+a^{2}b^{2}+b^{4}.
\end{equation}

\noindent \begin{definition}

The \emph{reductive factor,} $\mathrm{rf}\left(A,\,B\right)$, for
$A,B\in\mathrm{GL}\left(2,\,\mathbb{Q}\left[x,x^{-1}\right]\right)$
is the unique positive integer such that
\begin{align*}
\mathrm{rf}\left(A,\,B\right)\mathrm{mir}\left(AB\right) & =\mathrm{mir}\left(A\right)\mathrm{mir}\left(B\right).
\end{align*}
\end{definition}
\begin{lemma}
\begin{description}
\item [{(a)}] For any pair of integers $a,\,b$ such that $b\ne0$ and
$\gcd\left(a,\,b\right)=1$, we have
\begin{align*}
 & \mathrm{mir}\left(g\left[\frac{a}{b}\right]\right)=\pm b^{2}g\left[\frac{a}{b}\right],\\
 & \mathrm{mir}\left(g\left[\frac{a}{b}\right]^{-1}\right)=\pm b^{2}\det\left(g\left[\frac{a}{b}\right]\right)g\left[\frac{a}{b}\right]^{-1}.
\end{align*}
\item [{(b)}] For any integers $a,\,b,\,c,\,d$ such that $b\ne0\ne d$
and $\gcd\left(a,\,b\right)=1=\gcd\left(c,\,d\right)$, define
\begin{align*}
F_{a,\,b,\,c,\,d} & :=a^{2}c^{2}+abcd+b^{2}d^{2}.
\end{align*}
\end{description}
\noindent 
Then, we have
\begin{align*}
\mathrm{rf}\left(g\left[\frac{a}{b}\right],\,g\left[\frac{c}{d}\right]^{-1}\right) & =\gcd\left(F_{a,\,b,\,c,\,d},\,ad-bc\right).
\end{align*}

\begin{description}
\item [{(c)}] With the same assumption as (b), we have the following identity
\begin{align*}
 & \,\gcd\left(F_{a,\,b,\,c,\,d},\,ad-bc\right)\\
= & \,\gcd\left(b^{4}\det g\left[\frac{a}{b}\right],\,ad-bc\right)\\
= & \,\gcd\left(d^{4}\det g\left[\frac{c}{d}\right],\,ad-bc\right)\\
= & \,\gcd\left(b^{4}\det g\left[\frac{a}{b}\right],\,d^{4}\det g\left[\frac{c}{d}\right],\,ad-bc\right).
\end{align*}
\end{description}
\end{lemma}
\begin{proof}

\noindent (a) It is straightforward from (38) and the assumption $\gcd\left(a,\,b\right)=1$.

\noindent (b) Temporarily put $m$ as $ad-bc$ in this proof. Then,
we claim
\begin{align}
\mathrm{rf}\left(g\left[\frac{a}{b}\right],\,g\left[\frac{c}{d}\right]^{-1}\right) & =\gcd\left(adm,\,bcm,\,acm,\,bdm,\,F_{a,\,b,\,c,\,d}\right).
\end{align}
From (38) and (a), we expand
\begin{align*}
 & \,\left(b^{2}g\left[\frac{a}{b}\right]\right)\left(d^{2}\det\left(g\left[\frac{c}{d}\right]\right)g\left[\frac{c}{d}\right]^{-1}\right)\\
= & \,\left(\begin{array}{cc}
b^{2}t-a^{2} & ab\\
-\left(1+t^{-1}+t\right)ab & b^{2}t^{-1}-a^{2}
\end{array}\right)\left(\begin{array}{cc}
d^{2}t^{-1}-c^{2} & -cd\\
\left(1+t^{-1}+t\right)cd & d^{2}t-c^{2}
\end{array}\right)\\
= & \,\left(\begin{array}{cc}
\left(bc-ad\right)adt^{-1}+a^{2}c^{2}+b^{2}d^{2}+abcd+\left(ad-bc\right)bct & *\\
\left(ad-bc\right)ac+\left(ad-bc\right)bdt & *
\end{array}\right)^{T},
\end{align*}
where we have (44) by comparing the terms. We also claim
\begin{align}
m & =\gcd\left(adm,\,bcm,\,acm,\,bdm\right).
\end{align}

By simplifying the right-hand side of (45) repetitively, we prove
(45) as

\noindent 
\begin{align*}
 & \,\gcd\left(adm,\,bcm,\,acm,\,bdm\right)\\
= & \,m\gcd\left(ad,\,bc,\,ac,\,bd\right)\\
= & \,m\gcd\left(a,\,bc,\,bd\right)\\
= & \,m\gcd\left(a,\,b\right)\\
= & \,m.
\end{align*}

By (45), the equality (44) is reduced to
\begin{align*}
\mathrm{rf}\left(g\left[\frac{a}{b}\right],\,g\left[\frac{c}{d}\right]^{-1}\right) & =\gcd\left(m,\,F_{a,\,b,\,c,\,d}\right),
\end{align*}
which establishes (b).

\noindent (c) We use the convention on $m=ad-bc$. A direct computation
using (43) and the definition of $F_{a,\,b,\,c,\,d}$ yields

\noindent 
\begin{align}
 & a^{2}F_{a,\,b,\,c,\,d}\equiv c^{2}\left(b^{4}\det g\left[\frac{a}{b}\right]\right)\;(\mathrm{mod}\;m),\\
 & c^{2}F_{a,\,b,\,c,\,d}\equiv a^{2}\left(d^{4}\det g\left[\frac{c}{d}\right]\right)\;(\mathrm{mod}\;m).
\end{align}

Since $b^{4}\det g\left[\frac{a}{b}\right]$ (resp. $d^{4}\det g\left[\frac{c}{d}\right]$)
and $a^{2}$ (resp. $c^{2}$) has no common factor, from (46) and
(47), we have

\noindent 
\begin{align}
 & \gcd\left(m,\,b^{4}\det g\left[\frac{a}{b}\right]\right)\;|\;F_{a,\,b,\,c,\,d},\\
 & \gcd\left(m,\,d^{4}\det g\left[\frac{c}{d}\right]\right)\;|\;F_{a,\,b,\,c,\,d}.
\end{align}

On the other hand, since $\gcd\left(m,\,F_{a,\,b,\,c,\,d}\right)$
and $c$ (resp. $a$) has no common factor, from (46) (resp. (47)),
we also have

\noindent 
\begin{align}
 & \gcd\left(m,\,F_{a,\,b,\,c,\,d}\right)\;|\;b^{4}\det g\left[\frac{a}{b}\right],\\
 & \gcd\left(m,\,F_{a,\,b,\,c,\,d}\right)\;|\;d^{4}\det g\left[\frac{c}{d}\right].
\end{align}

We now have $\gcd\left(m,\,F_{a,\,b,\,c,\,d}\right)=\gcd\left(m,\,b^{4}\det g\left[\frac{a}{b}\right]\right)$,
by deducing from (48) and (50) that
\begin{align*}
 & \,\gcd\left(m,\,F_{a,\,b,\,c,\,d}\right)\\
= & \,\gcd\left(m,\,F_{a,\,b,\,c,\,d},\,\gcd\left(m,\,b^{4}\det g\left[\frac{a}{b}\right]\right)\right)\\
= & \,\gcd\left(m,\,F_{a,\,b,\,c,\,d},\,b^{4}\det g\left[\frac{a}{b}\right]\right)\\
= & \,\gcd\left(m,\,b^{4}\det g\left[\frac{a}{b}\right],\,\gcd\left(m,\,F_{a,\,b,\,c,\,d}\right)\right)\\
= & \,\gcd\left(m,\,b^{4}\det g\left[\frac{a}{b}\right]\right).
\end{align*}

From (49) and (51), we have $\gcd\left(m,\,F_{a,\,b,\,c,\,d}\right)=\gcd\left(m,\,d^{4}\det g\left[\frac{c}{d}\right]\right)$
in the same manner, which establishes (c). $\qedhere$

\noindent \end{proof}

\noindent \begin{definition}

Define a free reduced word $C$ in the set of elementary generators
$\left\{ \overline{g}\left[r\right]\right\} _{r\in\mathbb{Q}}$ to
be \emph{alternating} if each pair of adjacent symbols is of form
$\overline{g}\left[r_{1}\right]\overline{g}\left[r_{2}\right]^{-1}$
or $\overline{g}\left[r_{1}\right]^{-1}\overline{g}\left[r_{2}\right]$.
For a pair of positive integers $k,\:l$, we define a $\left(k,\,l\right)$-\emph{chain}
to be an alternating word in $\left\{ \overline{g}\left[r\right]\right\} _{r\in\mathbb{Q}}$,
inductively as follows.
\begin{description}
\item [{(a)}] Any length 1 element $\overline{g}\left[\frac{a}{b}\right]$
or $\overline{g}\left[\frac{a}{b}\right]^{-1}$ is a $\left(k,\,l\right)$-chain
for any positive integers $k,\:l$ such that $kl=b^{4}\det\left(g\left[\frac{a}{b}\right]\right)=a^{4}+a^{2}b^{2}+b^{4}$.
\item [{(b)}] Suppose $C=\overline{g}\left[\frac{a_{1}}{b_{1}}\right]\cdots\overline{g}\left[\frac{a_{n}}{b_{n}}\right]$
is a $\left(k,\,l\right)$-chain. Then, $C\overline{g}\left[\frac{a_{n+1}}{b_{n+1}}\right]^{-1}$
is a $\left(k',\,l'\right)$-chain if
\end{description}
\begin{itemize}
\item $k=k'$,
\item $l\;|\;\mathrm{rf}\left(g\left[\frac{a_{n}}{b_{n}}\right],\,g\left[\frac{a_{n+1}}{b_{n+1}}\right]^{-1}\right)$,
\item $ll'=b_{n+1}^{4}\det\left(g\left[\frac{a_{n+1}}{b_{n+1}}\right]\right)$,
and
\item $k'l'=\det\left(\mathrm{mir}\left(C\overline{g}\left[\frac{a_{n+1}}{b_{n+1}}\right]^{-1}\right)\right)$,
\end{itemize}
and the other alternating cases are also defined in the same manner.
Define the \emph{length} of a chian $C$ to be the word length of
$C$, and define a \emph{biminimal chain} to be a $\left(1,\,1\right)$-chain.

\noindent \end{definition}

\noindent \begin{conjecture}

Let $C$ be a $\left(1,\,l\right)$-chain of length $n$ such that
\begin{align*}
(n\,\mathrm{odd}) & \;C=\overline{g}\left[\frac{a_{1}}{b_{1}}\right]\cdots\overline{g}\left[\frac{a_{n}}{b_{n}}\right],\;\mathrm{or}\\
(n\;\mathrm{even}) & \;C=\overline{g}\left[\frac{a_{1}}{b_{1}}\right]\cdots\overline{g}\left[\frac{a_{n}}{b_{n}}\right]^{-1}.
\end{align*}

Suppose a pair of positive integers $c,\,d$ satisfies that $a_{n}d\ne b_{n}c$
and
\begin{align*}
l\;|\;\mathrm{rf}\left(g\left[\frac{a_{n}}{b_{n}}\right],\,g\left[\frac{c}{d}\right]^{-1}\right).
\end{align*}

Then, when $n$ is odd (resp. $n$ is even), $C\overline{g}\left[\frac{c}{d}\right]^{-1}$
(resp. $C\overline{g}\left[\frac{c}{d}\right]$) is a $\left(1,\,l^{-1}d^{4}\det g\left[\frac{c}{d}\right]\right)$-chain.

\noindent \end{conjecture}

Conjecture 5.5 plays a role in finding a $\left(1,\,l'\right)$-chain
to elongate another $\left(1,\,l\right)$-chain by calculating only
on the 'last' generator in the latter. This conjecture seems true,
but the proof may require deeper results, for example, on long-distance
reductions.

For any \emph{nontrivial} (i.e. not just one of $\overline{g}\left[0\right]^{\pm1}$)
biminimal chain, there exists a maximal one among the reductive factors
of two adjacent symbols. More concretely, let
\begin{align*}
C=\overline{g}\left[\frac{a_{1}}{b_{1}}\right]\overline{g}\left[\frac{a_{2}}{b_{2}}\right]^{-1}\cdots\overline{g}\left[\frac{a_{n+1}}{b_{n+1}}\right]^{-1}
\end{align*}
be a biminimal chain, and define $\mathrm{rf}_{i}$ to be $\mathrm{rf}\left(g\left[\frac{a_{i}}{b_{i}}\right],\,g\left[\frac{a_{i+1}}{b_{i+1}}\right]^{-1}\right)$.
Then, we define the \emph{maximal reductive factor} $\mathrm{Mrf}\left(C\right):=\max\left\{ \mathrm{rf}_{i}\right\} _{1\le i\le n}$.
Suppose $\mathrm{Mrf}\left(C\right)=\mathrm{rf}_{k}$ for some integer
$k$. Then, by the definition, we have
\begin{align*}
\mathrm{Mrf}\left(C\right)^{2} & \ge\max\left(b_{k}^{4}\det g\left[\frac{a_{k}}{b_{k}}\right],\,b_{k+1}^{4}\det g\left[\frac{a_{k+1}}{b_{k+1}}\right]\right).
\end{align*}

The following algorithm begins by searching for candidates for maximal
reductive factors.

\noindent \begin{algorithm}

For a positive integer $M$ as input, we generate a tree with nodes
of the following type:
\begin{align*}
\left(a,\,b,\,\mathrm{rf},\,\mathrm{index}\right),
\end{align*}
where $a,\,b,\,\mathrm{rf},\,\mathrm{index}$ are integers such that
$\mathrm{rf}\ge0$.
\begin{description}
\item [{(a)}] Insert a node $\left(0,\,0,\,0,\,0\right)$. It is the single
root node of the whole tree. Proceed to (b).
\item [{(b)}] Choose a pair of integers $\left(a,\,b,\,c,\,d\right)$ iteratively
such that $0<b,\,d\le M$, $0<\left|a\right|,\,\left|c\right|\le M$,
$b=\max\left(\left|a\right|,\,b,\,\left|c\right|,\,d\right)$, $\gcd\left(a,\,b\right)=1=\gcd\left(c,\,d\right)$,
and $\left(a,\,b\right)\ne\left(c,\,d\right)$. Then, proceed to (c).
If there is no combination $\left(a,\,b,\,c,\,d\right)$ remaining,
terminate the algorithm.
\item [{(c)}] Calculate the value $\mathrm{rf}\left(g\left[\frac{a}{b}\right],\,g\left[\frac{c}{d}\right]^{-1}\right)$
by Lemma 5.3 (b) and (c), and define an integer
\begin{align*}
\mathrm{rf}_{0} & :=\mathrm{rf}\left(g\left[\frac{a}{b}\right],\,g\left[\frac{c}{d}\right]^{-1}\right).
\end{align*}
Then, proceed to (d).
\item [{(d)}] Test if
\begin{align}
\mathrm{rf}_{0}^{2} & \ge\max\left(b^{4}\det g\left[\frac{a}{b}\right],\,d^{4}\det g\left[\frac{c}{d}\right]\right).
\end{align}
\item [{(d-1)}] If the condition (52) is not satisfied, then return; we
go back to (b) where we choose the next combination $\left(a,\,b,\,c,\,d\right)$.
\item [{(d-2)}] If the condition (52) is satisfied, insert two nodes
\begin{align*}
 & \left(a,\,b,\,\mathrm{rf}_{0},\,i\right),\\
 & \left(c,\,d,\,\mathrm{rf}_{0},\,-i\right),
\end{align*}
as children nodes of the root $\left(0,\,0,\,0,\,0\right)$, where
$i$ is a positive integer added by 1 to the largest value among the
indices of existing children of the root. If there was no child of
the root, take $i=1$. Then, test if $\mathrm{rf}_{0}=b^{4}\det g\left[\frac{a}{b}\right]$
for the node $\left(a,\,b,\,\mathrm{rf}_{0},\,i\right)$.
\item [{(d-2-1)}] If the test returns false, proceed to (e) for the tested
node $\left(a,\,b,\,\mathrm{rf}_{0},\,i\right)$ as input. After finishing
the procedure (e), test if there still exists the child node $\left(a,\,b,\,\mathrm{rf}_{0},\,i\right)$
of the root. If the test returns false, then delete the node $\left(c,\,d,\,\mathrm{rf}_{0},\,-i\right)$
and return; we go back to (b). If the test returns true, test if $\mathrm{rf}_{0}=d^{4}\det g\left[\frac{c}{d}\right]$.
\item [{(d-2-1-1)}] If the test returns true, then return; we go back to
(b).
\item [{(d-2-1-2)}] If the test returns false, proceed to (e) for the tested
node $\left(c,\,d,\,\mathrm{rf}_{0},\,-i\right)$ as input. After
finishing the procedure (e), test if there still exists the child
node $\left(c,\,d,\,\mathrm{rf}_{0},\,-i\right)$ of the root. If
the test returns true, then return; we go back to (b). If the test
returns false, then delete the node $\left(a,\,b,\,\mathrm{rf}_{0},\,i\right)$
and return; we go back to (b).
\item [{(d-2-2)}] If the test returns true, test if $\mathrm{rf}_{0}=d^{4}\det g\left[\frac{c}{d}\right]$.
\item [{(d-2-2-1)}] If the test returns true, then return; we go back to
(b).
\item [{(d-2-2-2)}] If the test returns false, proceed to (e) for the tested
node $\left(c,\,d,\,\mathrm{rf}_{0},\,-i\right)$ as input. After
finishing the procedure (e), test if there still exists the child
node $\left(c,\,d,\,\mathrm{rf}_{0},\,-i\right)$ of the root. If
the test returns true, then return; we go back to (b). If the test
returns false, then delete the node $\left(a,\,b,\,\mathrm{rf}_{0},\,i\right)$
and return; we go back to (b).
\item [{(e)}] For the input node $\left(e,\,f,\,\mathrm{rf},\,\mathrm{index}\right)$,
choose a pair of integers $\left(g,\,h\right)$ iteratively such that
$0<\left|g\right|\le M$, $0<h\le M$ such that $\gcd\left(g,\,h\right)=1$
and $\left(e,\,f\right)\ne\left(g,\,h\right)$, then proceed to (f).
If there is no combination $\left(g,\,h\right)$ remaining, test if
there is a child node of the input $\left(e,\,f,\,\mathrm{rf},\,\mathrm{index}\right)$.
\item [{(e-1)}] If the test returns true, then return; if the parent node
of $\left(e,\,f,\,\mathrm{rf},\,\mathrm{index}\right)$ is a node
with a nonzero index, we go back to (e) for the parent node of the
input $\left(e,\,f,\,\mathrm{rf},\,\mathrm{index}\right)$, where
we choose the next combination $\left(g,\,h\right)$. If the parent
node is the root node, we go back to (d-2-1-2) or (d-2-2-2).
\item [{(e-2)}] If the test returns false, delete the input node $\left(e,\,f,\,\mathrm{rf},\,\mathrm{index}\right)$,
and return; if the parent node of $\left(e,\,f,\,\mathrm{rf},\,\mathrm{index}\right)$
is a node with a nonzero index, we go back to (e) for the parent node
of $\left(e,\,f,\,\mathrm{rf},\,\mathrm{index}\right)$. If the parent
node is the root node, we go back to (d-2-1-2) or (d-2-2-2).
\item [{(f)}] By Lemma 5.3 (b), (c), calculate the values $\mathrm{rf}\left(g\left[\frac{e}{f}\right],\,g\left[\frac{g}{h}\right]^{-1}\right)$
and $\mathrm{rf}^{-1}f^{4}\det g\left[\frac{e}{f}\right]$. Define
integers
\begin{align*}
 & \mathrm{rf}_{near}:=\mathrm{rf}\left(g\left[\frac{e}{f}\right],\,g\left[\frac{g}{h}\right]^{-1}\right),\\
 & \mathrm{rf}_{child}:=\mathrm{rf}^{-1}f^{4}\det g\left[\frac{e}{f}\right].
\end{align*}
Then, proceed to (g).
\item [{(g)}] Test the following three conditions:
\begin{itemize}
\item $\mathrm{rf}_{child}\;|\;\mathrm{rf}_{near}$,
\item $\mathrm{rf}_{child}<\mathrm{rf}_{0}$, where $\mathrm{rf}_{0}$ is
the $\mathrm{rf}$ of the ancestor node whose parent is the root,
and
\item $\mathrm{rf}_{child}\ne\mathrm{rf}_{anc}$ for any $\mathrm{rf}_{anc}$
which is $\mathrm{rf}$ of an ancestor.
\end{itemize}
\item [{(g-1)}] If at least one condition is not satisfied, then return;
we go back to (e) for the original node $\left(e,\,f,\,\mathrm{rf},\,\mathrm{index}\right)$.
\item [{(g-2)}] If the three conditions are satisfied, insert a new node
\begin{align*}
\left(g,\,h,\,\mathrm{rf}_{child},\,\mathrm{index}\right)
\end{align*}
as a child of $\left(e,\,f,\,\mathrm{rf},\,\mathrm{index}\right)$.
Test if $h^{4}\det g\left[\frac{g}{h}\right]=\mathrm{rf}_{near}$.
If the test returns true, then return; we go back to (e) for the original
node $\left(e,\,f,\,\mathrm{rf},\,\mathrm{index}\right)$. If the
test returns false, then here start a new search (e) for the child
$\left(g,\,h,\,\mathrm{rf}_{child},\,\mathrm{index}\right)$, and
after finishing the new procedure (e) for the child, return; we go
back to (e) for the original node $\left(e,\,f,\,\mathrm{rf},\,\mathrm{index}\right)$.
\end{description}
\noindent \end{algorithm}

In Algorithm 5.6, the input value $M$ determines the maximum number
of iterations allowed, given in steps (b) and (e). At first, we observe
that Algorithm 5.6 will terminate in finite time, due to the condition
$\mathrm{rf}_{child}\ne\mathrm{rf}_{anc}$ for any rf of an ancestor
in step (g). If this condition is omitted, the algorithm may potentially
loop endlessly. After the algorithm finishes, we obtain candidates
for nontrivial biminimal chains in the form of a tree. If Conjecture
5.5 is true, these candidates are indeed biminimal chains. The condition
$\mathrm{rf}_{child}<\mathrm{rf}_{0}$ in step (g) ensures that the
factor $\mathrm{rf}_{0}$ found satisfying (52) is indeed the maximal
reductive factor in the chain. If this condition is omitted, we will
only find \emph{locally} maximal factors, and the algorithm may produce
the same two chains with different indices, resulting in inefficiency.

\noindent \begin{prooftheorem11}

\noindent By using Algorithm 5.6 with the Rust Programming Language,
when we used an input $M=50$, we found an element $\left[\begin{array}{cc}
g_{1} & g_{2}\\
-\Phi\overline{g_{2}} & \overline{g_{1}}
\end{array}\right]\in\mathrm{PGL}\left(2,\,\mathbb{Q}\left[t,\,t^{-1}\right]\right)$ as a word of elementary generators
\begin{align*}
 & \overline{g}\left[-1\right]\overline{g}\left[2\right]^{-1}\overline{g}\left[-\frac{1}{3}\right]\overline{g}\left[4\right]^{-1}\,\cdots\,\overline{g}\left[-\frac{1}{9}\right]\overline{g}\left[-\frac{11}{8}\right]^{-1}\overline{g}\left[\frac{13}{17}\right]\overline{g}\left[-\frac{18}{19}\right]^{-1}\\
 & \overline{g}\left[\frac{25}{9}\right]\overline{g}\left[\frac{8}{31}\right]^{-1}\overline{g}\left[-\frac{41}{17}\right]\overline{g}\left[\frac{31}{46}\right]^{-1}\overline{g}\left[-\frac{38}{43}\right]\overline{g}\left[\frac{29}{37}\right]^{-1}\overline{g}\left[-\frac{20}{23}\right]\\
 & \overline{g}\left[-21\right]^{-1}\overline{g}\left[\frac{1}{20}\right]\overline{g}\left[-19\right]^{-1}\,\cdots\,\overline{g}\left[\frac{1}{2}\right]\overline{g}\left[-1\right]^{-1},
\end{align*}
where the relative degree of $g_{1}$ is 40, and the maximal reductive
factor is
\begin{align*}
3081 & =\mathrm{rf}\left(g\left[\frac{31}{46}\right]^{-1},\,g\left[-\frac{38}{43}\right]\right)
\end{align*}
such that
\begin{align*}
3081^{2} & \ge\max\left\{ 46^{4}\det g\left[\frac{31}{46}\right],\,43^{4}\det g\left[-\frac{38}{43}\right]\right\} .
\end{align*}

This element is a candidate of biminimal chain, and one can directly
check that this is in fact a biminimal chain included in $\mathrm{PGL}\left(2,\,\mathbb{Z}\left[t,\,t^{-1}\right]\right)$.
Therefore, the integral subgroup $U$ includes it. Corollary 4.7 implies
that Question 2.2 is false. Moreover, the example above assures that
the deviating subgroup $F$ is nontrivial. By the proof of Theorem
3.17, we also have
\begin{align*}
\left[\overline{\Gamma}\::\:\overline{\beta}\left(B_{3}\right)\right] & =\left|F\right|=\infty,
\end{align*}
which concludes the proof of Theorem 1.1. $\qed$

\noindent \end{prooftheorem11}

When we use Algorithm 5.6 with an input $M=121$, we found 9 distinct
biminimal chains, each of which cannot be reduced to a product of
others. Among these, we confirmed that the smallest relative degree
of representatives is 22. This proves that the free rank of the group
$F$ is at least 9. Therefore, it is natural to pose the following
conjecture.

\noindent \begin{conjecture}

The deviating subgroup $F$ is free of countably infinite rank.

\noindent \end{conjecture}

While we will not add a detailed explanation here, analyzing the group
$F$ can play a role in approaching the problem of faithfulness of
the Burau representation for $n=4$, which is still open. If $F$
were trivial, this fact could be indirect evidence for the faithfulness.

\noindent \begin{remark}

Let $i$ be the usual imaginary unit. Define a group $\mathcal{Q}_{i}$
in the same manner as Definition 4.3, by extending the base field
$\mathbb{Q}$ to $\mathbb{Q}\left(i\right)$. Also, define a group
$U_{i}$ by
\begin{align*}
U_{i} & :=\mathcal{Q}_{i}\cap\mathrm{PGL}\left(2,\,\mathbb{Z}\left[i\right]\left[t,\,t^{-1}\right]\right).
\end{align*}

Observe that the equation $x^{4}+x^{2}+1=0$ has no solution over
the field $\mathbb{Q}\left(i\right)$, which implies that the determinant
of an elementary generator $g\left[r\right]$ for each $r\in\mathbb{Q}\left(i\right)$
is nonzero. Therefore, we see that $\mathcal{Q}_{i}$ is also freely
generated by the elementary generators, as in the proof of ${\mathrm{Theorem}\;4.9}$
in Section 4. In addition, we mention without proof that $U_{i}$
is free of countably infinite rank, supporting Conjecture 5.7. The
key idea of the proof is that for a Pythagorean triple $\left(a,\,b,\,c\right)$
such that $0<a<b<c$, we have
\begin{align*}
a^{4}\det g\left[\frac{ic}{a}\right] & =c^{4}\det g\left[\frac{ib}{c}\right],
\end{align*}
and the triples $\left(a,\,b,\,c\right)$ such that $c=b+1$ are parametrized
by
\begin{align*}
\left(a_{n},\,b_{n},\,c_{n}\right) & =\left(2n+1,\,2n^{2}+2n,\,2n^{2}+2n+1\right),\;n>0.
\end{align*}

Based on these facts, we found infinitely many biminimal chains in
$U_{i}$, and the largest absolute value of the denominator $b$ in
a symbol $\overline{g}\left[\frac{a}{b}\right]$ in a chain is unbounded.

\noindent \end{remark}

\begin{spacing}{0.9}
\bibliographystyle{amsplain}
\phantomsection\addcontentsline{toc}{section}{\refname}\bibliography{bibgen}

\end{spacing}

$ $

{\small{}Donsung Lee; \href{mailto:disturin@snu.ac.kr}{disturin@snu.ac.kr}}{\small\par}

{\small{}Department of Mathematical Sciences and Research Institute
of Mathematics,}{\small\par}

{\small{}Seoul National University, Gwanak-ro 1, Gwankak-gu, Seoul,
South Korea 08826}{\small\par}

\clearpage{}

\pagebreak{}

\pagenumbering{arabic}

\renewcommand{\thefootnote}{A\arabic{footnote}}
\renewcommand{\thepage}{A\arabic{page}}
\renewcommand{\thetable}{A\arabic{table}}
\renewcommand{\thefigure}{A\arabic{figure}}

\setcounter{footnote}{0} 
\setcounter{section}{0}
\setcounter{table}{0}
\setcounter{figure}{0}
\end{document}